\documentclass[11pt,reqno]{amsart}

\usepackage{amsxtra}
\usepackage[all]{xy}
\usepackage{amssymb}
\usepackage{amsmath}
\usepackage{amsfonts}
\usepackage{amscd}
\usepackage{mathtools}
\usepackage{mathrsfs}
\usepackage{tikz-cd}
\usepackage{tikz}
\usepackage{array}
\usepackage{booktabs}
\usepackage{graphicx}
\usepackage{float}
\usepackage{enumitem}
\usepackage{cases}
\usepackage{appendix}
\usepackage{microtype}
\usepackage{iftex}

\ifPDFTeX
  \usepackage{txfonts}
\fi
\usepackage{dsfont}

\usepackage[colorlinks=true,final,backref=page,hyperindex]{hyperref}
\hypersetup{
  linkcolor=red,
  citecolor=green,
  urlcolor=blue
}

\theoremstyle{plain}
\newtheorem{theorem}{Theorem}[section]
\newtheorem{proposition}[theorem]{Proposition}
\newtheorem{lemma}[theorem]{Lemma}
\newtheorem{corollary}[theorem]{Corollary}
\newtheorem{definition}[theorem]{Definition}
\newtheorem{example}[theorem]{Example}
\newtheorem{remark}[theorem]{Remark}

\newcommand{\kk}{\mathbb{K}}
\newcommand{\g}{\mathfrak{g}}
\newcommand{\h}{\mathfrak{h}}
\newcommand{\U}{\mathcal{U}}
\newcommand{\F}{\mathcal{F}}
\newcommand{\id}{\mathrm{id}}

\newcommand{\Der}{\mathrm{Der}}

\begin{document}

\title{Universal Enveloping Algebras of Hom-Lie Algebras with \texorpdfstring{$\alpha$}{alpha}-Derivations}

\author{Zhangqiuyu Jiang} \address[Zhangqiuyu Jiang]{School of Mathematical Sciences\\
Zhejiang Normal University\\
Jinhua 321004, China}
\email{jiangzhangqiuyu2002@zjnu.edu.cn}
\author{Chuangchuang Kang}
\address[Chuangchuang Kang]{School of Mathematical Sciences\\
Zhejiang Normal University\\
Jinhua 321004, China}
\email{kangcc@zjnu.edu.cn}
\author{Jiafeng L\"u}
\address[Jiafeng L\"u]{
School of Mathematical Sciences    \\
Zhejiang Normal University\\
Jinhua 321004              \\
China}
\email{jiafenglv@zjnu.edu.cn}
\thanks{*Corresponding Author: Chuangchuang Kang. Email: kangcc@zjnu.edu.cn.}

\begin{abstract}
In this paper, we construct nonunital universal enveloping Hom-associative algebras for multiplicative Hom-Lie algebras equipped with $\alpha$-derivations. We prove that the prescribed derivation extends to the enveloping algebra and that the resulting universal property yields a left adjoint to the commutator functor between the corresponding categories. When the twisting map is bijective, we establish a canonical isomorphism between this enveloping algebra and the twist of the nonunital universal enveloping algebra of the associated untwisted Lie algebra. This isomorphism preserves the twisting maps and the derivations. Consequently, the classical Poincar\'e-Birkhoff-Witt theorem identifies the associated graded algebra for the untwisted product with the positive-degree symmetric algebra. When the Hom-Lie algebra is finite-dimensional, the resulting PBW basis consists of ordered monomials of positive length. The twisting map, the extended $\alpha$-derivation, and its untwisted counterpart preserve the PBW filtration.
\end{abstract}



\maketitle

\tableofcontents
\section{Introduction}

	Universal enveloping algebras relate Lie algebras to associative algebras
	through the commutator construction. They identify representations of a Lie
	algebra with modules over an associative algebra. Moreover, the
	Poincar\'e-Birkhoff-Witt (PBW) theorem describes the associated graded
	algebra as the symmetric algebra and gives an ordered monomial basis
	\cite{Humphreys}. These relations motivate the study of enveloping
	constructions that preserve additional algebraic structures.

	Hom-Lie algebras were introduced in the study of deformations defined by
	twisted derivations \cite{HartwigLarssonSilvestrov}. Their relation to
	Hom-associative algebras is given by the commutator construction
	\cite{MakhloufSilvestrov}, while the twisting construction of Yau associates
	Hom-algebras to ordinary algebras equipped with endomorphisms
	\cite{HomAlgebraTwists}. Gohr studied conditions under which
	Hom-associative structures with surjective twisting maps admit
	associative untwistings \cite{Gohr}.
	Cohomology and deformation theory provide further
	approaches to Hom-algebras \cite{AmmarEjbehiMakhlouf}, including
	$\alpha$-type Hochschild cohomology \cite{HurleMakhlouf}.
	For Hom-Lie algebras, Sheng introduced $\alpha^k$-derivations
	\cite{AlphaKDerivations}, and Li and Wang studied Hom-Lie algebras with
	derivations \cite{LiWang}. The relation between derivations,
	low-dimensional cohomology, abelian extensions, and crossed modules
	was studied by Casas and Garc\'ia-Mart\'inez \cite{CasasGarciaMartinez}.

	In this paper, we study the enveloping problem for a multiplicative Hom-Lie
	algebra equipped with a prescribed $\alpha$-derivation. Such a derivation
	commutes with the twisting map and satisfies a twisted Leibniz rule.
	The enveloping problem therefore asks for an extension of this derivation
	together with a universal property in which morphisms preserve it.
	When the twisting map is bijective, a further question is how the
	enveloping construction relates to untwisting and to the classical PBW
	filtration. In the classical setting, enveloping algebras for
	(modified) $\lambda$-differential Lie algebras were constructed by
	Peng, Zhang, Gao, and Luo \cite{PengZhangGaoLuo}.

	Free Hom-algebras provide a starting point for the enveloping problem.
	Yau constructed enveloping Hom-associative algebras using weighted trees
	\cite{HomLieEnveloping}, and Hellstr\"om, Makhlouf, and Silvestrov studied
	free Hom-associative algebras and the corresponding enveloping problem
	\cite{HellstromMakhloufSilvestrov}. In the involutive setting, free
	Hom-associative algebras were constructed by Zheng and Guo \cite{ZhengGuo}.
	Other enveloping constructions include the Hom-Hopf structure of
	Laurent-Gengoux, Makhlouf, and Teles \cite{LaurentGengouxMakhloufTeles}
	and the regular Hom-Poisson case studied by Hu \cite{Hu}.
	For the PBW problem, Guo, Zhang, and Zheng treated involutive Hom-Lie
	algebras \cite{GuoZhangZheng}. An enveloping construction and a PBW
	theorem for involutive color Hom-Lie algebras were obtained by Armakan,
	Silvestrov, and Farhangdoost \cite{ArmakanSilvestrovFarhangdoost}.
	Bai, Chen, and Zhang established a
	PBW theorem for multiplicative regular triangularizable Hom-Lie algebras
	\cite{BaiChenZhang}. A related result in representation theory is
	an analogue of the Ado theorem established by Makhlouf and Zusmanovich
	\cite{MakhloufZusmanovich} for finite-dimensional nilpotent regular
	Hom-Lie algebras over algebraically closed fields of characteristic zero.

	We work in the categories $\mathbf{HomLieDer}_\alpha$ and
	$\mathbf{HomAsDer}_\alpha$, whose morphisms preserve the algebraic
	operations, the twisting maps, and the chosen derivations. The
	commutator construction defines a functor
	$HLie_{\Der_\alpha}:\mathbf{HomAsDer}_\alpha\to
	\mathbf{HomLieDer}_\alpha$
	(Propositions~\ref{prop:commutator} and~\ref{taishezhijiandeduiying}).
	For each object $(\g,[\cdot,\cdot],\alpha,d)$ of
	$\mathbf{HomLieDer}_\alpha$, we construct an enveloping Hom-associative
	algebra $(\U_{\mathrm m}(\g),\mu_U,\alpha_U,D_U)$ carrying an extension
	of $d$ (Proposition~\ref{prop:Um-HomAsDer}). We prove that it satisfies
	the universal property in these categories (Theorem~\ref{thm:main}).
	Consequently, the enveloping functor is left adjoint to
	$HLie_{\Der_\alpha}$ (Corollary~\ref{cor:functor-adjunction}).

	The construction uses the free Hom-nonassociative algebra of
	\cite{HomLieEnveloping} and its multiplicative quotient.
	The algebra structure on a direct sum in Proposition~\ref{prop:semidirect-HNA}
	gives the unique extension of $d$
	to an $\alpha$-derivation on the free multiplicative
	Hom-nonassociative algebra. We then show that the ideal generated by
	the Hom-associativity and bracket--commutator relations is stable under
	both the twisting map and the extended derivation. Thus these maps
	descend to the enveloping algebra.

	Suppose now that $\alpha$ is bijective. Let
	$\g_0=(\g,[\cdot,\cdot]_0)$, where
	$[x,y]_0=\alpha^{-1}([x,y])$, and set $\delta=\alpha^{-1}\circ d$.
	Then $(\g_0,\delta)$ is a Lie algebra with a derivation
	(Proposition~\ref{prop:reg-lie}).
	We denote by $\U^+(\g_0)$ its nonunital universal enveloping
	algebra, and by $\widetilde{\alpha}$ and $\widetilde{\delta}$ the
	extensions of $\alpha$ and $\delta$, respectively. Define
	$a\cdot_\alpha b=\widetilde{\alpha}(ab)$ and
	$\widetilde d=\widetilde{\alpha}\circ\widetilde{\delta}$.
	We prove that there is a canonical isomorphism
	\[
	(\U_{\mathrm m}(\g),\mu_U,\alpha_U,D_U)
	\cong
	(\U^+(\g_0),\cdot_\alpha,\widetilde{\alpha},\widetilde d)
	\]
	in $\mathbf{HomAsDer}_\alpha$
	(Theorem~\ref{thm:regular-comparison}). In particular, this comparison
	preserves the twisting maps and the prescribed derivations.

	The Theorem~\ref{thm:regular-comparison} allows us to apply the classical PBW theorem
	to the untwisted product $u*v=\alpha_U^{-1}(uv)$ on
	$\U_{\mathrm m}(\g)$. For this product, the associated graded algebra
	is $S^+(\g_0)=\bigoplus_{p\geq1}S^p(\g_0)$. If $\g$ is
	finite-dimensional, the ordered monomials of positive length in the
	images of an ordered basis of $\g$, formed with respect to $*$,
	form a basis of $\U_{\mathrm m}(\g)$. Moreover, $\alpha_U$,
	$\Delta_U:=\alpha_U^{-1}\circ D_U$, and $D_U$ preserve the PBW
	filtration (Theorem~\ref{cor:regular-pbw}).
	The enveloping and untwisting constructions in the regular case are
	summarized in Figure~\ref{fig:regular-framework}.

	\begin{figure}[H]
		\centering
		\resizebox{0.7\linewidth}{!}{%
			\begin{tikzcd}[
				ampersand replacement=\&,
				row sep=4.8em,
				column sep=8.5em
				]
				(\g,[\cdot,\cdot],\alpha,d)
				\arrow[
				r,
				"\substack{
					\text{Untwisting}\\
					\text{Proposition~\ref{prop:reg-lie}}}"
				]
				\arrow[
				d,
				"\substack{
					\text{Enveloping}\\
					\text{Theorem~\ref{thm:main}}}"'
				]
				\&
				(\g_0,\delta)
				\arrow[
				d,
				"\substack{
					\text{Enveloping + twisting}\\
					\text{Proposition~\ref{prop:classical-yau-twist}}}"
				]
				\\
				(\U_{\mathrm m}(\g),\mu_U,\alpha_U,D_U)
				\arrow[
				r,
				no head,
				double,
				double distance=1.1pt,
				"\substack{
					\text{Canonical isomorphism}\\
					\text{Theorem~\ref{thm:regular-comparison}}}",
				"\cong"'
				]
				\arrow[
				d,
				"\substack{
					\text{Untwisting}\\
					\text{Lemma~\ref{lem:regularity-envelope} and}\\
					\text{Proposition~\ref{prop:reg-ass}}}"'
				]
				\&
				(\U^{+}(\g_0),\cdot_{\alpha},
				\widetilde{\alpha},\widetilde d)
				\arrow[
				d,
				"\substack{
					\text{Untwisting}\\
					\text{Proposition~\ref{prop:reg-ass}}}"
				]
				\\
				(\U_{\mathrm m}(\g),*,\Delta_U)
				\arrow[
				r,
				no head,
				double,
				double distance=1.1pt,
				"\substack{
					\text{Canonical isomorphism}\\
					\text{Theorem~\ref{thm:regular-comparison}}}",
				"\cong"'
				]
				\&
				(\U^{+}(\g_0),\cdot,\widetilde{\delta})
			\end{tikzcd}%
		}
		\caption{Enveloping algebras and untwisting in the regular case.}
		\label{fig:regular-framework}
	\end{figure}
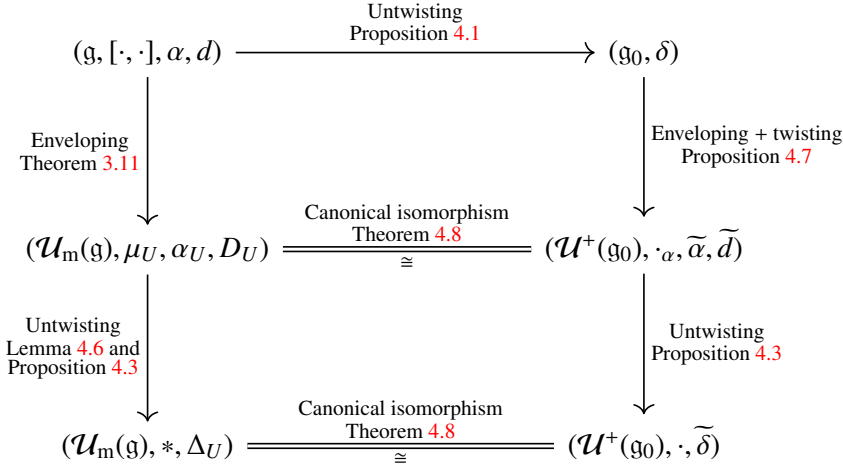
	The paper is organized as follows. In Section~2, we recall the relevant
	definitions and introduce the categories and the commutator functor.
	In Section~3, we construct the enveloping algebra and prove its universal
	property. Section~4 establishes the comparison theorem and the PBW theorem
	in the regular case. Section~5 summarizes the results and records further
	problems.

	Throughout the paper, all vector spaces are defined over a fixed field $\kk$
	of characteristic $0$. Unless otherwise stated, Hom-Lie,
	Hom-associative, and Hom-nonassociative algebras are assumed to be
	multiplicative. We follow the conventions of \cite{AlphaKDerivations}
	for Hom-Lie algebras and \cite{MakhloufSilvestrov} for
	Hom-associative algebras. The term \emph{regular} means that the
	twisting map is bijective. Hom-associative and Hom-nonassociative
	algebras are taken to be nonunital unless otherwise stated, and
	morphisms are not required to preserve units.

\section{Preliminaries on Hom-Lie and Hom-associative algebras with \texorpdfstring{$\alpha$}{alpha}-derivations}
	
	In this section, we recall Hom-Lie and Hom-associative algebras,
	their $\alpha$-derivations and modules, and the corresponding categories.
	We give constructions using direct sums and twisting and describe the commutator
	functor.
	
	\subsection{Definitions and examples}
	
	\begin{definition}[{\cite[Definition~2.1]{AlphaKDerivations}}]
		A \textbf{Hom-Lie algebra} is a triple $(\g,[\cdot,\cdot],\alpha)$,
		where $\g$ is a $\kk$-vector space,
		$[\cdot,\cdot]:\g\times\g\to\g$ is a skew-symmetric bilinear map,
		and $\alpha:\g\to\g$ is a linear map satisfying the Hom-Jacobi identity
		$$
		[\alpha(x),[y,z]]+[\alpha(y),[z,x]]+[\alpha(z),[x,y]]=0, \qquad \forall x,y,z\in\g, 
		$$
		and
		\begin{equation}
			\alpha([x,y])=[\alpha(x),\alpha(y)], \qquad \forall x,y\in\g. \label{Liemultiplicativity}
		\end{equation}
		Identity~\eqref{Liemultiplicativity} is called the
		\textbf{multiplicativity condition}. A Hom-Lie algebra is called
		\textbf{regular} if $\alpha$ is bijective.
	\end{definition}
	
	\begin{remark}
		We use the term non-multiplicative Hom-Lie algebra when the
		multiplicativity condition is not imposed.
	\end{remark}
	
	\begin{definition}
		Let $(\g,[\cdot,\cdot],\alpha)$ be a Hom-Lie algebra. A linear map $d:\g\to\g$
		is called an \textbf{$\alpha$-derivation} if $d\circ\alpha=\alpha\circ d$
		and
		$$
		d([x,y])=[d(x),\alpha(y)]+[\alpha(x),d(y)], \qquad \forall x,y\in\g.
		$$
		The quadruple $(\g,[\cdot,\cdot],\alpha,d)$ is then called a
		\textbf{Hom-Lie algebra with an $\alpha$-derivation}.
	\end{definition}
	
	\begin{remark}
		This is the case $k=1$ of the $\alpha^k$-derivations introduced in
		\cite{AlphaKDerivations}.
	\end{remark}
	
	\begin{definition}[{\cite[Definition~1.1]{MakhloufSilvestrov}}]\label{def:hom-ass-nonass}
		A \textbf{Hom-nonassociative algebra} is a triple $(A,\mu,\alpha)$,
		where $A$ is a $\kk$-vector space, $\mu:A\times A\to A$ is a bilinear
		map, written $\mu(a,b)=ab$, and $\alpha:A\to A$ is a linear map such that
		\begin{equation}
		\alpha(ab)=\alpha(a)\alpha(b), \qquad \forall a,b\in A. \label{jiehemultiplicativity}
		\end{equation}
		Identity~\eqref{jiehemultiplicativity} is called the
		\textbf{multiplicativity condition}. A \textbf{Hom-associative algebra}
		is a Hom-nonassociative algebra satisfying the Hom-associativity identity
		$$
		\alpha(a)(bc)=(ab)\alpha(c), \qquad \forall a,b,c\in A. 
		$$
		It is called \textbf{regular} if $\alpha$ is bijective.
	\end{definition}
	
	\begin{remark}
		We use the term non-multiplicative Hom-nonassociative algebra when
		the multiplicativity condition is not imposed.
	\end{remark}
	
	Let $(A,\mu,\alpha)$ be a Hom-nonassociative algebra. Define the linear
	map $\widetilde{\alpha}:A\oplus A\to A\oplus A$ by
	$$
	\widetilde{\alpha}(a,u):=(\alpha(a),\alpha(u)),
	\qquad \forall (a,u)\in A\oplus A,
	$$
	and the bilinear map $\star:(A\oplus A)\times(A\oplus A)\to A\oplus A$ by
	$$
	(a,u)\star(b,v):=
	\bigl(ab,\,u\alpha(b)+\alpha(a)v\bigr),
	\qquad \forall (a,u),(b,v)\in A\oplus A.
	$$
	
	\begin{proposition}\label{prop:semidirect-HNA}
		The triple $(A\oplus A,\star,\widetilde{\alpha})$ is a Hom-nonassociative algebra.
	\end{proposition}
	
	\begin{proof}
		The map $\widetilde{\alpha}$ is linear and $\star$ is bilinear by
		construction. To verify multiplicativity, let $(a,u),(b,v)\in A\oplus A$.
		Using multiplicativity of $\alpha$, we obtain
		\begin{align*}
			\widetilde{\alpha}((a,u)\star(b,v))
			&=\widetilde{\alpha}(ab,u\alpha(b)+\alpha(a)v)\\
			&=(\alpha(ab),\alpha(u\alpha(b)+\alpha(a)v))\\
			&=(\alpha(a)\alpha(b),\alpha(u)\alpha^2(b)+\alpha^2(a)\alpha(v))\\
			&=(\alpha(a),\alpha(u))\star(\alpha(b),\alpha(v))\\
			&=\widetilde{\alpha}(a,u)\star \widetilde{\alpha}(b,v).
		\end{align*}
		Therefore, $(A\oplus A,\star,\widetilde{\alpha})$ is a
		Hom-nonassociative algebra.
	\end{proof}
	
	Let $(\g,[\cdot,\cdot],\alpha)$ be a Hom-Lie algebra. Define the linear
	map $\widetilde{\alpha}:\g\oplus\g\to\g\oplus\g$ by
	$$
	\widetilde{\alpha}(x,u):=(\alpha(x),\alpha(u)), \qquad \forall (x,u)\in\g\oplus\g,
	$$
	and the bilinear map $[\cdot,\cdot]_{\g\oplus\g}:(\g\oplus\g)\times(\g\oplus\g)\to\g\oplus\g$ by
	$$
	[(x,u),(y,v)]_{\g\oplus\g}:=
	\bigl([x,y],\,[u,\alpha(y)]+[\alpha(x),v]\bigr), \qquad \forall (x,u),(y,v)\in\g\oplus\g.
	$$
	
	\begin{proposition}\label{prop:HomLie-direct-sum}
		The triple $(\g\oplus\g,[\cdot,\cdot]_{\g\oplus\g},\widetilde{\alpha})$ is a Hom-Lie algebra.
	\end{proposition}
	
	\begin{proof}
		The bracket is bilinear and skew-symmetric. For
		$(x,u),(y,v),(z,w)\in\g\oplus\g$, multiplicativity follows from
		\begin{align*}
			\widetilde{\alpha}([(x,u),(y,v)]_{\g\oplus\g})
			&=\bigl(\alpha([x,y]),\alpha([u,\alpha(y)]+[\alpha(x),v])\bigr)\\
			&=\bigl([\alpha(x),\alpha(y)],
			[\alpha(u),\alpha^2(y)]+[\alpha^2(x),\alpha(v)]\bigr)\\
			&=[\widetilde{\alpha}(x,u),\widetilde{\alpha}(y,v)]_{\g\oplus\g}.
		\end{align*}
		For the Hom-Jacobi identity, we compute
		\begin{align*}
			&[\widetilde{\alpha}(x,u),[(y,v),(z,w)]_{\g\oplus\g}]_{\g\oplus\g}
			+[\widetilde{\alpha}(y,v),[(z,w),(x,u)]_{\g\oplus\g}]_{\g\oplus\g}\\
			&\quad
			+[\widetilde{\alpha}(z,w),[(x,u),(y,v)]_{\g\oplus\g}]_{\g\oplus\g}\\
			&=
			\left(
			\begin{array}{@{}l@{}}
				[\alpha(x),[y,z]]
				+[\alpha(y),[z,x]]
				+[\alpha(z),[x,y]],\\[2pt]
				\begin{aligned}
					&[\alpha(u),[\alpha(y),\alpha(z)]]
					+[\alpha^2(x),[v,\alpha(z)]]
					+[\alpha^2(x),[\alpha(y),w]]\\
					&+[\alpha(v),[\alpha(z),\alpha(x)]]
					+[\alpha^2(y),[w,\alpha(x)]]
					+[\alpha^2(y),[\alpha(z),u]]\\
					&+[\alpha(w),[\alpha(x),\alpha(y)]]
					+[\alpha^2(z),[u,\alpha(y)]]
					+[\alpha^2(z),[\alpha(x),v]]
				\end{aligned}
			\end{array}
			\right)\\
			&=
			\left(
			\begin{array}{@{}l@{}}
				0,\\[2pt]
				\begin{aligned}
					&\bigl([\alpha(u),[\alpha(y),\alpha(z)]]
					+[\alpha^2(y),[\alpha(z),u]]
					+[\alpha^2(z),[u,\alpha(y)]]\bigr)\\
					&+\bigl([\alpha(v),[\alpha(z),\alpha(x)]]
					+[\alpha^2(z),[\alpha(x),v]]
					+[\alpha^2(x),[v,\alpha(z)]]\bigr)\\
					&+\bigl([\alpha(w),[\alpha(x),\alpha(y)]]
					+[\alpha^2(x),[\alpha(y),w]]
					+[\alpha^2(y),[w,\alpha(x)]]\bigr)
				\end{aligned}
			\end{array}
			\right)\\
			&=(0,0).
		\end{align*}
		Here the first component and each parenthesized expression in the second
		component vanish by the Hom-Jacobi identity in $\g$. Therefore,
		$(\g\oplus\g,[\cdot,\cdot]_{\g\oplus\g},\widetilde{\alpha})$ is a
		Hom-Lie algebra.
	\end{proof}

	\begin{definition}\label{daidaozidehomnonassociativedaishu}
		Let $(A,\mu,\alpha)$ be a Hom-nonassociative algebra. A linear map $D:A\to A$
		is called an \textbf{$\alpha$-derivation} if $D\circ\alpha=\alpha\circ D$
		and
		$$
		D(ab)=D(a)\alpha(b)+\alpha(a)D(b), \qquad \forall a,b\in A.
		$$
		The quadruple $(A,\mu,\alpha,D)$ is then called a
		\textbf{Hom-nonassociative algebra with an $\alpha$-derivation}.
	\end{definition}
	
	\begin{definition}\label{def:hom-ideal}
		Let $(A,\mu,\alpha)$ be a Hom-nonassociative algebra. A
		\textbf{Hom-ideal} of $(A,\mu,\alpha)$ is a two-sided ideal $I$ of
		$(A,\mu)$ such that $\alpha(I)\subseteq I$.
	\end{definition}
	
	The following construction extends the twisting of Lie algebras to include
	a derivation.

	\begin{proposition}\label{diyigelizi}
		Let $(\g,[\cdot,\cdot]_0)$ be a Lie algebra, let $\delta:\g\to\g$
		be a derivation, and let $\alpha:\g\to\g$
		be a Lie algebra endomorphism such that
		$\alpha\circ\delta=\delta\circ\alpha$.
		Then $(\g,\alpha\circ[\cdot,\cdot]_0,\alpha)$ is a Hom-Lie algebra,
		and $\alpha\circ\delta$ is an $\alpha$-derivation of this algebra.
	\end{proposition}
	
	\begin{proof}
		Set $[\cdot,\cdot]_\alpha:=\alpha\circ[\cdot,\cdot]_0$ and
		$d:=\alpha\circ\delta$. By \cite[Corollary~2.5]{HomAlgebraTwists},
		$(\g,[\cdot,\cdot]_\alpha,\alpha)$ is a Hom-Lie algebra. The relation
		$\alpha\delta=\delta\alpha$ gives
		$d\alpha=\alpha\delta\alpha=\alpha^2\delta=\alpha d$.
		Since $\delta$ is a derivation of $(\g,[\cdot,\cdot]_0)$, for
		$x,y\in\g$ we obtain
		\begin{align*}
			d([x,y]_\alpha)
			&=
			\alpha\delta(\alpha([x,y]_0))
			=
			\alpha^2\delta([x,y]_0)\\
			&= \alpha^2([\delta(x),y]_0+[x,\delta(y)]_0)\\
			&= \alpha([\alpha\delta(x),\alpha(y)]_0+[\alpha(x),\alpha\delta(y)]_0)\\
			&= [\alpha\delta(x),\alpha(y)]_\alpha+[\alpha(x),\alpha\delta(y)]_\alpha\\
			&= [d(x),\alpha(y)]_\alpha+[\alpha(x),d(y)]_\alpha.
		\end{align*}
		Thus $(\g,[\cdot,\cdot]_\alpha,\alpha,d)$ is a Hom-Lie algebra with an
		$\alpha$-derivation.
	\end{proof}

	\begin{example}
		Let $\h$ be the Heisenberg Lie algebra over $\kk$ with basis $x,y,z$
		and brackets $[x,y]_0=z$, $[x,z]_0=[y,z]_0=0$. Define
		$\alpha,\delta:\h\to\h$ by
		\[
		\begin{aligned}
		\alpha(x)&=2x, &\alpha(y)&=3y, &\alpha(z)&=6z,\\
		\delta(x)&=x, &\delta(y)&=2y, &\delta(z)&=3z.
		\end{aligned}
		\]
		By Proposition~\ref{diyigelizi}, $(\h,[\cdot,\cdot]_\alpha,\alpha,d)$
		is a regular Hom-Lie algebra with an $\alpha$-derivation, where
		$[u,v]_\alpha:=\alpha([u,v]_0)$ and $d:=\alpha\circ\delta$.
		Explicitly,
		\[
		[x,y]_\alpha=6z,\qquad
		d(x)=2x,\quad d(y)=6y,\quad d(z)=18z.
		\]
	\end{example}
	
	The same construction applies to associative algebras.

	\begin{proposition}
		Let $(A,\ast)$ be an associative algebra, let $\Delta:A\to A$ be a
		derivation, and let $\alpha:A\to A$ be an algebra endomorphism satisfying
		$\alpha\Delta=\Delta\alpha$.
		Then $(A,\alpha\circ\ast,\alpha)$ is a Hom-associative algebra,
		and $\alpha\circ\Delta$ is an $\alpha$-derivation of this algebra.
	\end{proposition}
	
	\begin{proof}
		Set $\mu(a,b):=\alpha(a\ast b)$ for $a,b\in A$ and
		$D:=\alpha\circ\Delta$. By \cite[Corollary~2.5]{HomAlgebraTwists},
		$(A,\mu,\alpha)$ is a Hom-associative algebra. The relation
		$\alpha\Delta=\Delta\alpha$ gives $D\alpha=\alpha D$.
		Together with the Leibniz rule for $\Delta$, it also gives
		$D(ab)=D(a)\alpha(b)+\alpha(a)D(b)$.
	\end{proof}

	\subsection{Hom-modules and Hom-Lie (associative) modules}\label{subsec:modules}
	
	\begin{definition}[{\cite[Section~3.1]{HomLieEnveloping}}]\label{def:hom-module}
		A \textbf{Hom-module} is a pair $(V,\alpha_V)$ consisting of a
		$\kk$-vector space $V$ and a linear map $\alpha_V:V\to V$. A
		\textbf{morphism} of Hom-modules from $(V,\alpha_V)$ to
		$(W,\alpha_W)$ is a linear map $f:V\to W$ such that
		$f\circ\alpha_V=\alpha_W\circ f$.
	\end{definition}
	
	\begin{definition}[{\cite[Definition~4.1]{AlphaKDerivations}}]\label{def:hom-lie-module}
		Let $(\g,[\cdot,\cdot],\alpha)$ be a Hom-Lie algebra. A
		\textbf{Hom-Lie module} over $\g$ is a Hom-module $(V,\alpha_V)$ equipped
		with a linear map $\rho:\g\to \operatorname{End}_{\kk}(V)$
		such that
		$$
		\rho(\alpha(x))\circ\alpha_V=\alpha_V\circ\rho(x), \qquad \forall x\in\g, 
		$$
		and
		$$
		\rho([x,y])\circ\alpha_V
		=\rho(\alpha(x))\circ\rho(y)-\rho(\alpha(y))\circ\rho(x), \qquad \forall x,y\in\g. 
		$$
	\end{definition}
	
    \begin{definition}\label{def:hom-ass-module}
        Let $(A,\mu,\alpha)$ be a Hom-associative algebra. A
        \textbf{left Hom-associative module} over $A$ is a Hom-module
        $(M,\alpha_M)$ equipped with a bilinear map
        $A \times M\to M, (a, m)\mapsto a\cdot m,$ satisfying
        $$
        \alpha_M(a\cdot m)=\alpha(a)\cdot\alpha_M(m), \qquad \forall a\in A,\ m\in M, 
        $$
        and
        $$
        (ab)\cdot\alpha_M(m)=\alpha(a)\cdot(b\cdot m), \qquad \forall a,b\in A,\ m\in M.
        $$
    \end{definition}

	\subsection{Categories and the commutator functor}
	
	\begin{definition}
		We denote by $\mathbf{HomLieDer}_\alpha$ the category whose objects are Hom-Lie
		algebras $(\g,[\cdot,\cdot],\alpha,d)$ with an $\alpha$-derivation. A
		morphism
		$f:(\g,[\cdot,\cdot],\alpha,d)\to
		(\h,[\cdot,\cdot]_{\h},\beta,\delta)$ is a linear map $f:\g\to\h$
		satisfying
		$$
		f([x,y])=[f(x),f(y)]_{\h},\qquad
		f\circ\alpha=\beta\circ f,\qquad
		f\circ d=\delta\circ f, \qquad \forall x,y\in\g. 
		$$
	\end{definition}
	
	\begin{definition}
		We denote by $\mathbf{HomAsDer}_\alpha$ the category whose objects are
		Hom-associative algebras $(A,\mu,\alpha,D)$ with an
		$\alpha$-derivation. A morphism
		$F:(A,\mu,\alpha,D)\to(B,\nu,\beta,E)$ is a linear map $F:A\to B$
		satisfying
		$$
		F(\mu(a,b))=\nu(F(a),F(b)),\qquad
		F\circ\alpha=\beta\circ F,\qquad
		F\circ D=E\circ F, \qquad \forall a,b\in A. 
		$$
	\end{definition}
	
	Let $(A,\mu,\alpha,D)$ be an object of $\mathbf{HomAsDer}_\alpha$.
	Define the bilinear map $[\cdot,\cdot]_A:A\times A\to A$ by
	$$
	[a,b]_A:=ab-ba, \qquad \forall a,b\in A.
	$$
	
	\begin{proposition}\label{prop:commutator}
		The quadruple $(A,[\cdot,\cdot]_A,\alpha,D)$ is an object of $\mathbf{HomLieDer}_\alpha$.
	\end{proposition}
	
	\begin{proof}
		By \cite[Proposition~1.6]{MakhloufSilvestrov},
		$(A,[\cdot,\cdot]_A,\alpha)$ is a Hom-Lie algebra. It remains to verify
		that $D$ is an $\alpha$-derivation of the commutator bracket. The relation
		$D\alpha=\alpha D$ holds because $(A,\mu,\alpha,D)$ is an object of
		$\mathbf{HomAsDer}_\alpha$. For any $a,b\in A$, the twisted Leibniz rule gives
		\begin{align*}
			D([a,b]_A)
			&=D(ab-ba)\\
			&=D(ab)-D(ba)\\
			&=D(a)\alpha(b)+\alpha(a)D(b)
			-D(b)\alpha(a)-\alpha(b)D(a)\\
			&=(D(a)\alpha(b)-\alpha(b)D(a))+(\alpha(a)D(b)-D(b)\alpha(a))\\
			&=[D(a),\alpha(b)]_A+[\alpha(a),D(b)]_A.
		\end{align*}
		Therefore, $(A,[\cdot,\cdot]_A,\alpha,D)$ is an object of
		$\mathbf{HomLieDer}_\alpha$.
	\end{proof}
	
	\begin{proposition}\label{taishezhijiandeduiying}
		If $F:(A,\mu,\alpha,D)\to(B,\nu,\beta,E)$ is a morphism in
		$\mathbf{HomAsDer}_\alpha$, then $F$ is also a morphism
		$(A,[\cdot,\cdot]_A,\alpha,D)\to
		(B,[\cdot,\cdot]_B,\beta,E)$ in $\mathbf{HomLieDer}_\alpha$.
	\end{proposition}
	
	\begin{proof}
		For all $a,b\in A$, we have
		$$
		F([a,b]_A)=F(ab-ba)=F(a)F(b)-F(b)F(a)=[F(a),F(b)]_B.
		$$
		Since $F$ is a morphism in $\mathbf{HomAsDer}_\alpha$, we also have
		$F\circ\alpha=\beta\circ F$ and $F\circ D=E\circ F$. Hence $F$ is a
		morphism in $\mathbf{HomLieDer}_\alpha$.
	\end{proof}
	
	By Propositions~\ref{prop:commutator} and~\ref{taishezhijiandeduiying},
	the commutator construction applies to both objects and morphisms.
	We denote this assignment by $HLie_{\Der_\alpha}$. Since identity maps
	and compositions are unchanged, we obtain the following result.
	
	\begin{proposition}
		The assignment $HLie_{\Der_\alpha}$ is a functor from
		$\mathbf{HomAsDer}_\alpha$ to $\mathbf{HomLieDer}_\alpha$.
	\end{proposition}

	\section{Universal enveloping algebras of Hom-Lie algebras with \texorpdfstring{$\alpha$}{alpha}-derivations}
	
	In this section, we construct the universal enveloping algebra of a Hom-Lie
	algebra with an $\alpha$-derivation. We first construct the free
	Hom-nonassociative algebra by imposing multiplicativity on a weighted-tree
	algebra. We then extend the derivation and pass to a Hom-associative quotient.
	
	\subsection{The free non-multiplicative Hom-nonassociative algebra}
	
	We recall the weighted-tree notation of \cite[Section~2]{HomLieEnveloping}
	and use it to describe the free Hom-nonassociative algebra without imposing
	multiplicativity, following \cite[Theorem~1]{HomLieEnveloping}.
	
	For $n\ge 1$, let $T_n$ be the set of planar binary trees with $n$ leaves and
	one root. An element of $T_n$ is called an $n$-tree, and its leaves are labeled
	from left to right by $x_1,\ldots,x_n$. We denote the unique element of $T_1$
	by $i$. If $\psi\in T_p$ and $\varphi\in T_q$, where $p,q\ge 1$, their
	\emph{grafting} is the $(p+q)$-tree
	$\psi\vee\varphi\in T_{p+q}$ obtained by joining the two roots to a new root.
	
	For $n=1,2,3$, we have
	$$
	T_1=\{i\}, \qquad T_2=\{i\vee i\}, \qquad T_3=\{(i\vee i)\vee i,
	\; i\vee(i\vee i)\}.
	$$
	The trees are depicted below. The two elements of $T_3$ correspond to the
	products $(x_1x_2)x_3$ and $x_1(x_2x_3)$, respectively.
	
	
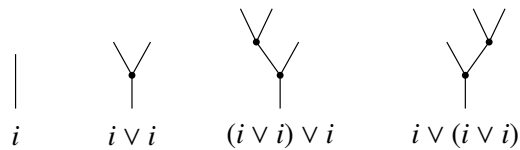
\begin{center}
		\begin{tikzpicture}[x=0.7cm,y=0.8cm,line width=0.45pt]
			\draw (-5,0)--(-5,0.9);
			\node at (-5,-0.45) {$i$};
			
			\draw (-2.8,0)--(-2.8,0.55);
			\fill (-2.8,0.55) circle (1.3pt);
			\draw (-2.8,0.55)--(-3.15,1.1);
			\draw (-2.8,0.55)--(-2.45,1.1);
			\node at (-2.8,-0.45) {$i\vee i$};
			
			\draw (0,0)--(0,0.55);
			\fill (0,0.55) circle (1.3pt);
			\draw (0,0.55)--(-0.45,1.1);
			\draw (0,0.55)--(0.35,1.1);
			\fill (-0.45,1.1) circle (1.3pt);
			\draw (-0.45,1.1)--(-0.75,1.65);
			\draw (-0.45,1.1)--(-0.15,1.65);
			\node at (0,-0.45) {$(i\vee i)\vee i$};
			
			\draw (3.5,0)--(3.5,0.55);
			\fill (3.5,0.55) circle (1.3pt);
			\draw (3.5,0.55)--(3.15,1.1);
			\draw (3.5,0.55)--(3.95,1.1);
			\fill (3.95,1.1) circle (1.3pt);
			\draw (3.95,1.1)--(3.65,1.65);
			\draw (3.95,1.1)--(4.25,1.65);
			\node at (3.5,-0.45) {$i\vee(i\vee i)$};
		\end{tikzpicture}
	\end{center}
	
	A weighted $n$-tree is a pair $\tau=(\psi,w)$, where $\psi\in T_n$ and $w$
	assigns a nonnegative integer to each internal vertex of $\psi$. We denote the
	set of weighted $n$-trees by $T_n^{\mathrm{wt}}$. Since the $1$-tree has no
	internal vertex, $T_1^{\mathrm{wt}}=T_1$. For an internal vertex $v$, the
	integer $w(v)$ records how many times the twisting map is applied at $v$.
	
	If $\tau\in T_p^{\mathrm{wt}}$ and $\sigma\in T_q^{\mathrm{wt}}$, their
	weighted grafting $\tau\vee\sigma\in T_{p+q}^{\mathrm{wt}}$ is obtained by
	grafting the underlying trees and assigning weight $0$ to the new lowest
	internal vertex. For $m\ge 0$ and $\tau\in T_n^{\mathrm{wt}}$ with $n\ge 2$,
	we write $\tau[m]$ for the tree obtained by adding $m$ to the weight of the
	lowest internal vertex, with all other weights unchanged. Thus
	every weighted $n$-tree with $n\ge 2$ admits a unique decomposition
	$\tau=(\tau_1\vee\tau_2)[r]$, where
	$\tau_1\in T_p^{\mathrm{wt}}$, $\tau_2\in T_q^{\mathrm{wt}}$, $p+q=n$, and
	$r$ is the weight of the lowest internal vertex; see
	\cite[Sections~2.3--2.5]{HomLieEnveloping}.
	
	For example, the following diagram represents $\tau=((i\vee i)[r]\vee i)[s].$
	
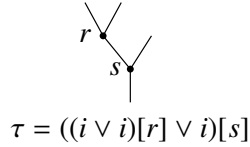
	\begin{center}
		\begin{tikzpicture}[x=0.8cm,y=0.8cm,line width=0.45pt]
			\draw (0,0)--(0,0.55);
			\fill (0,0.55) circle (1.3pt);
			\draw (0,0.55)--(-0.45,1.1);
			\draw (0,0.55)--(0.35,1.1);
			\fill (-0.45,1.1) circle (1.3pt);
			\draw (-0.45,1.1)--(-0.75,1.65);
			\draw (-0.45,1.1)--(-0.15,1.65);
			\node[left] at (-0.48,1.1) {$r$};
			\node[left] at (0.02,0.55) {$s$};
			\node at (0,-0.45) {$\tau=((i\vee i)[r]\vee i)[s]$};
		\end{tikzpicture}
	\end{center}
	
	Let $(A,\mu_A,\alpha_A)$ be a Hom-nonassociative algebra without a
	multiplicativity assumption. Given a weighted tree $\tau$ and elements
	$a_1,\ldots,a_n\in A$, we define $(a_1\cdots a_n)_\tau\in A$ recursively.
	For $n=1$, set $(a)_i=a$. If $\tau=(\tau_1\vee\tau_2)[r]$ with $\tau_1\in T_p^{\mathrm{wt}}$, then
	$$
	(a_1\cdots a_n)_\tau
	=\alpha_A^r((a_1\cdots a_p)_{\tau_1}(a_{p+1}\cdots a_n)_{\tau_2}).
	$$
	Thus we evaluate the two subtrees, multiply their values, and apply
	$\alpha_A^r$ at the lowest internal vertex. For example,
	$$
	(ab)_{(i\vee i)[r]}=\alpha_A^r(ab),
	$$
	$$
	(abc)_{((i\vee i)[r]\vee i)[s]}
	=\alpha_A^s(\alpha_A^r(ab)c),
	$$
	$$
	(abc)_{(i\vee(i\vee i)[r])[s]}
	=\alpha_A^s(a\alpha_A^r(bc)),
	$$
	and, for the weighted tree drawn above,
	$$
	(a_1a_2a_3)_\tau=\alpha_A^s(\alpha_A^r(a_1a_2)a_3).
	$$
	The underlying tree therefore determines the parenthesization, and the
	weights determine the positions and powers of $\alpha_A$.
	
	With this notation, we recall the free construction of
	\cite[Theorem~1]{HomLieEnveloping}.
	
	Let $(V,\alpha_V)$ be a Hom-module. For each $\tau\in T_n^{\mathrm{wt}}$,
	let $V^{\otimes n}_{\tau}$ be a copy of $V^{\otimes n}$. For
	$x_1,\ldots,x_n\in V$, write $(x_{1,n})_{\tau}:=(x_1\otimes\cdots\otimes x_n)_{\tau} \in V^{\otimes
	n}_{\tau}$. Set
	$$
	F_{\mathrm{HNAs}}(V):=
	\bigoplus_{n\ge 1}\bigoplus_{\tau\in T_n^{\mathrm{wt}}}V^{\otimes n}_{\tau}.
	$$
	Thus $F_{\mathrm{HNAs}}(V)$ is spanned by products of elements of $V$ indexed
	by weighted trees.
	
	Define a bilinear product $\mu_F$ by
	$$
	\mu_F((x_{1,n})_\tau,(x_{n+1,n+m})_\sigma)
	=(x_{1,n+m})_{\tau\vee\sigma},
	$$
	and define a linear map $\alpha_F$ by
	$$
	\alpha_F|_V=\alpha_V,
	\qquad
	\alpha_F((x_{1,n})_\tau)=(x_{1,n})_{\tau[1]}
	\quad (n\ge 2).
	$$
	Multiplication corresponds to grafting. On a word of length at least two,
	$\alpha_F$ increases the weight of the lowest internal vertex by $1$.
	We denote the natural inclusion by $i:V\to F_{\mathrm{HNAs}}(V)$.
	
	\begin{proposition}[{\cite[Theorem~1]{HomLieEnveloping}}]\label{prop:FHNAs-recall}
		Let $(V,\alpha_V)$ be a Hom-module. Then
		$(F_{\mathrm{HNAs}}(V),\mu_F,\alpha_F)$ is the free non-multiplicative
		Hom-nonassociative algebra on $(V,\alpha_V)$. More precisely, for every
		non-multiplicative Hom-nonassociative algebra $(A,\mu_A,\alpha_A)$ and
		every Hom-module morphism $f:V\to A$, there exists a unique
		Hom-nonassociative algebra morphism $g:F_{\mathrm{HNAs}}(V)\to A$ such that $g\circ i=f$.
	\end{proposition}
	
    \subsection{The free Hom-nonassociative algebra and its \texorpdfstring{$\alpha$}{alpha}-derivations}

	We now impose multiplicativity on the free algebra. Let $M_1$ be the
	two-sided ideal of $(F_{\mathrm{HNAs}}(V),\mu_F)$ generated by all elements
		$$
		\alpha_F(uv)-\alpha_F(u)\alpha_F(v),
		\qquad \forall u,v\in F_{\mathrm{HNAs}}(V).
		$$
		Define
		$$
		M_{n+1}:=\langle M_n\cup \alpha_F(M_n)\rangle,
		\qquad
		M_\infty:=\bigcup_{n\ge 1}M_n.
		$$
		Set
		$$
		\F_{\mathrm{m}}(V):=F_{\mathrm{HNAs}}(V)/M_\infty,
		\qquad
		\pi:F_{\mathrm{HNAs}}(V)\to \F_{\mathrm{m}}(V)
		$$
	where $\pi$ is the quotient map. The product and twisting map on the quotient
	are induced by $\mu_F$ and $\alpha_F$, respectively:
	$$
	\mu_{\mathrm{m}}(\pi(u),\pi(v)):=\pi(\mu_F(u,v)),
	\qquad
	\alpha_{\mathrm{m}}(\pi(u)):=\pi(\alpha_F(u)).
	$$
		
	\begin{lemma}\label{lem:multiplicativity-quotient}
		With the notation above, $(\F_{\mathrm{m}}(V),\mu_{\mathrm{m}},\alpha_{\mathrm{m}})$
		is a Hom-nonassociative algebra. 
	\end{lemma}
	
	\begin{proof}
		First, we show that $M_\infty$ is a two-sided ideal satisfying
		$\alpha_F(M_\infty)\subseteq M_\infty$.
		Since $M_1$ is a two-sided ideal and $M_{n+1}$ is the two-sided ideal generated by
		$M_n\cup\alpha_F(M_n)$, each $M_n$ is a two-sided ideal and
		$M_n\subseteq M_{n+1}$. Hence $M_\infty=\bigcup_{n\ge 1}M_n$ is a two-sided ideal. For $x\in
		M_\infty$, there exists $n\ge 1$ such that $x\in M_n$. Then
		$$
		\alpha_F(x)\in \alpha_F(M_n)\subseteq M_{n+1}\subseteq M_\infty.
		$$
		Thus $\alpha_F(M_\infty)\subseteq M_\infty$, so
		$\mu_{\mathrm{m}}$ and $\alpha_{\mathrm{m}}$ are well defined.
		
		The generators of $M_1$ give
		$\pi(\alpha_F(uv))=\pi(\alpha_F(u)\alpha_F(v))$. By the definitions of
		$\mu_{\mathrm{m}}$ and $\alpha_{\mathrm{m}}$, we obtain
		$$
		\alpha_{\mathrm{m}}(\pi(u)\pi(v))
		=
		\alpha_{\mathrm{m}}(\pi(u))\alpha_{\mathrm{m}}(\pi(v)).
		$$
		Therefore, $(\F_{\mathrm{m}}(V),\mu_{\mathrm{m}},\alpha_{\mathrm{m}})$ is a
		Hom-nonassociative algebra.
	\end{proof}

	Write $j:=\pi\circ i:V\to\F_{\mathrm{m}}(V)$ for the canonical Hom-module map.
	
	\begin{proposition}\label{prop:free-mult-HNA}
		The triple $(\F_{\mathrm{m}}(V),\mu_{\mathrm{m}},\alpha_{\mathrm{m}})$ is
		the free Hom-nonassociative algebra on $(V,\alpha_V)$. More precisely, for every Hom-nonassociative
		algebra $(A,\mu_A,\alpha_A)$ and every Hom-module morphism $f:V\to A$,
		there exists a unique morphism of Hom-nonassociative algebras
		$\overline f:\F_{\mathrm{m}}(V)\to A$ such that
		$\overline f\circ j=f$.
	\end{proposition}
	
	\begin{proof}
		We prove the universal property using the following commutative diagram.
		\begin{figure}[H]
			\centering
			\small
			\begin{tikzpicture}[
				>=stealth,
				every node/.style={inner sep=2pt}
				]
				\node (V) at (0,3) {$(V,\alpha_V)$};
				\node (FHNAs) at (0,1.5) {$(F_{\mathrm{HNAs}}(V),\mu_F,\alpha_F)$};
				\node (Fm) at (0,0) {$(\F_{\mathrm{m}}(V),\mu_{\mathrm{m}},\alpha_{\mathrm{m}})$};
				\node (A) at (7.4,1.5) {$(A,\mu_A,\alpha_A)$};
				
				\draw[->] (V) -- node[left] {$i$} (FHNAs);
				\draw[->] (FHNAs) -- node[left] {$\pi$} (Fm);
				\draw[->] (V.east) -- node[above] {$f$} (A.north west);
				\draw[->] (FHNAs.east) -- node[above] {$g$} (A.west);
				\node at (3.7,1.26) {\scriptsize Proposition~\ref{prop:FHNAs-recall}};
				\draw[->,dashed] (Fm.east) -- node[midway,below,yshift=-5pt] {$\exists!\,\overline f$} (A.south west);
			\end{tikzpicture}
			\caption{Universal property of the free Hom-nonassociative algebra.}
		\end{figure}
		
		By Proposition~\ref{prop:FHNAs-recall}, there exists a unique morphism of
		non-multiplicative Hom-nonassociative algebras $g:F_{\mathrm{HNAs}}(V)\to A$ such that $g\circ
		i=f$. Since $\alpha_A$ is multiplicative, for
		$u,v\in F_{\mathrm{HNAs}}(V)$ we have
		$$
		g(\alpha_F(uv)-\alpha_F(u)\alpha_F(v))
		=
		\alpha_A(g(u)g(v))-\alpha_A(g(u))\alpha_A(g(v))
		=0.
		$$
		Hence $g(M_1)=0$. Using $g \circ \alpha_F = \alpha_A \circ g$, we obtain by induction $g(M_n)=0$
		for all $n\ge 1$. Thus $g(M_\infty)=0$, and $g$ induces a unique linear map
		$\overline f:\F_{\mathrm{m}}(V)\to A$ such that $\overline f(\pi(u))=g(u)$ for $u\in
		F_{\mathrm{HNAs}}(V)$. Equivalently, $g=\overline f\circ\pi$. The induced map $\overline f$
		preserves products and twisting maps and satisfies $\overline f\circ j=f$.
		Its uniqueness follows from the surjectivity of $\pi$ and the uniqueness
		of $g$.
	\end{proof}

	For the remainder of this section, let
	$(\g,[\cdot,\cdot],\alpha,d)$ be an object of
	$\mathbf{HomLieDer}_\alpha$. We use the free Hom-nonassociative algebra
	$(\F_{\mathrm{m}}(\g),\mu_{\mathrm{m}},\alpha_{\mathrm{m}})$ on the Hom-module $(\g,\alpha)$ and
	write $j:\g\to \F_{\mathrm{m}}(\g)$ for its canonical Hom-module map. The next proposition extends
	$d$ to an
	$\alpha_{\mathrm m}$-derivation on this free algebra.
	
	\begin{proposition}\label{lem:extension-derivation-free}
		Let $(\g,[\cdot,\cdot],\alpha,d)$ be an object of
		$\mathbf{HomLieDer}_{\alpha}$. Let
		$(\F_{\mathrm{m}}(\g),\mu_{\mathrm{m}},\alpha_{\mathrm{m}})$ be the
		free Hom-nonassociative algebra of Proposition~\ref{prop:free-mult-HNA},
		with canonical Hom-module morphism
		$j:\g\to\F_{\mathrm{m}}(\g)$. Then there exists a unique linear map
		$\widehat d:\F_{\mathrm{m}}(\g)\to \F_{\mathrm{m}}(\g)$ such that
		\begin{align}
			\widehat d\circ j &= j\circ d, \label{eq:dhat-on-generators}\\
			\widehat d\circ \alpha_{\mathrm{m}} &= \alpha_{\mathrm{m}}\circ \widehat d, \label{eq:dhat-alpha}\\
			\widehat d(xy) &= \widehat d(x)\alpha_{\mathrm{m}}(y)+\alpha_{\mathrm{m}}(x)\widehat d(y),
			\qquad x,y\in \F_{\mathrm{m}}(\g). \label{eq:dhat-leibniz}
		\end{align}
	\end{proposition}
	
	\begin{proof}
		We construct the extension using the universal property and the algebra
		structure on the direct sum in Proposition~\ref{prop:semidirect-HNA}. The maps involved
		are shown in the following commutative diagram.
		\begin{figure}[H]
			\centering
			\begin{tikzpicture}[>=stealth,line width=0.45pt]
				\node (A) at (0,3) {$(\g,\alpha)$};
				\node (B) at (7.4,3) {$(\F_{\mathrm{m}}(\g),\mu_{\mathrm{m}},\alpha_{\mathrm{m}})$};
				\node (C) at (0,0) {$(\F_{\mathrm{m}}(\g),\mu_{\mathrm{m}},\alpha_{\mathrm{m}})$};
				\node (D) at (5.0,1.25) {$(\F_{\mathrm{m}}(\g)\oplus\F_{\mathrm{m}}(\g),\star,\widetilde{\alpha}_{\mathrm{m}})$};
				
				\draw[->] (A) -- node[above] {$j$} (B);
				\draw[->] (A) -- node[left] {$j$} (C);
				\draw[->] (A) -- node[above right] {$\eta$} (D);
				\draw[->,dashed] (C) -- node[below right] {$\Phi$} (D);
				\draw[->] (D) -- node[pos=0.55,above,sloped,yshift=2pt] {$p_1$} (B);
				\draw[->]
				([xshift=6pt]C.east)
				.. controls (3.4,-0.55) and (9.50,0.01) ..
				node[pos=0.82,right,xshift=4pt] {$\id_{\F_{\mathrm{m}}(\g)}$}
				([xshift=-6pt]B.south);
			\end{tikzpicture}
			\caption{Extension of the derivation to the free Hom-nonassociative algebra.}
			\label{fig:extension-derivation}
		\end{figure}
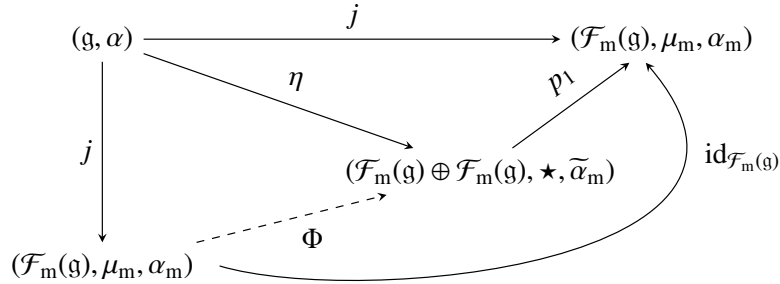
		
		Consider the Hom-nonassociative algebra
		$(\F_{\mathrm{m}}(\g)\oplus \F_{\mathrm{m}}(\g),\star,
		\widetilde{\alpha}_{\mathrm{m}})$ of
		Proposition~\ref{prop:semidirect-HNA}. For
		$(u,r),(v,s)\in \F_{\mathrm{m}}(\g)\oplus \F_{\mathrm{m}}(\g)$, its
		product and twisting map are
		\begin{align*}
		(u,r)\star(v,s)&:=(uv,r\alpha_{\mathrm{m}}(v)+\alpha_{\mathrm{m}}(u)s),\\
		\widetilde{\alpha}_{\mathrm{m}}(u,r)&:=(\alpha_{\mathrm{m}}(u),\alpha_{\mathrm{m}}(r)).
		\end{align*}
		Let $\eta:\g\to \F_{\mathrm{m}}(\g)\oplus \F_{\mathrm{m}}(\g)$ be the linear map given by
		$\eta(x):=(j(x),j(d(x)))$ for all $x\in\g$. Since $d\alpha=\alpha d$ and $j$ is a Hom-module
		morphism,
		$$
		\widetilde{\alpha}_{\mathrm{m}}(\eta(x))=(j(\alpha(x)),j(d(\alpha(x))))=\eta(\alpha(x)).
		$$
		Hence $\eta$ is a Hom-module morphism. By
		Proposition~\ref{prop:free-mult-HNA}, there exists a unique morphism of
		Hom-nonassociative algebras
		$\Phi:\F_{\mathrm{m}}(\g)\to
		\F_{\mathrm{m}}(\g)\oplus \F_{\mathrm{m}}(\g)$
		such that $\Phi\circ j=\eta$.
		
		Let $p_1:\F_{\mathrm{m}}(\g)\oplus \F_{\mathrm{m}}(\g)\longrightarrow \F_{\mathrm{m}}(\g)$ be the
		first projection, given by $p_1(u,v):=u$ for all $u,v\in \F_{\mathrm{m}}(\g)$. Then $p_1$ is a
		morphism of Hom-nonassociative algebras, and $p_1\Phi$
		satisfies $(p_1\Phi)j=j$. By the uniqueness in
		Proposition~\ref{prop:free-mult-HNA},
		$p_1\Phi=\id_{\F_{\mathrm{m}}(\g)}$. Hence, for every
		$u\in \F_{\mathrm{m}}(\g)$, there is a unique element
		$\widehat d(u)\in \F_{\mathrm{m}}(\g)$ such that
		$\Phi(u)=(u,\widehat d(u))$.
		The linearity of $\Phi$ gives a linear map
		$\widehat d:\F_{\mathrm{m}}(\g)\to \F_{\mathrm{m}}(\g)$.
		
		Taking second components in $\Phi(j(x))=\eta(x)$, we obtain
		$\widehat d(j(x))=j(d(x))$, which proves
		\eqref{eq:dhat-on-generators}. The compatibility of $\Phi$ with the
		twisting maps gives
		$$
		\Phi(\alpha_{\mathrm{m}}(u))=\widetilde{\alpha}_{\mathrm{m}}(\Phi(u))=(\alpha_{\mathrm{m}}(u),\alpha_{\mathrm{m}}(\widehat d(u))),
		$$
		whose second components give \eqref{eq:dhat-alpha}. Similarly, since
		$\Phi$ preserves products,
		$$
		\Phi(uv)=\Phi(u)\star\Phi(v)=(uv,\widehat d(u)\alpha_{\mathrm{m}}(v)+\alpha_{\mathrm{m}}(u)\widehat d(v)),
		$$
		which gives \eqref{eq:dhat-leibniz} by comparing second components.
		
		To prove uniqueness, let
		$\delta:\F_{\mathrm{m}}(\g)\to\F_{\mathrm{m}}(\g)$ be a linear map satisfying
		\eqref{eq:dhat-on-generators}--\eqref{eq:dhat-leibniz}. Then
		$u\mapsto (u,\delta(u))$ is a morphism of Hom-nonassociative algebras
		$\F_{\mathrm{m}}(\g)\to
		\F_{\mathrm{m}}(\g)\oplus \F_{\mathrm{m}}(\g)$ extending $\eta$. By
		the uniqueness of $\Phi$, this morphism equals $\Phi$. Therefore,
		$\delta=\widehat d$.
	\end{proof}
	
	\begin{remark}
		By Proposition~\ref{lem:extension-derivation-free},
		$(\F_{\mathrm{m}}(\g),\mu_{\mathrm{m}},\alpha_{\mathrm{m}},\widehat d)$
		is a Hom-nonassociative algebra with an $\alpha_{\mathrm m}$-derivation in the sense of
		Definition~\ref{daidaozidehomnonassociativedaishu}.
	\end{remark}
	
\subsection{From free Hom-nonassociative algebras to Hom-associative algebras}
	
	\begin{lemma}\label{lem:ideal-closure}
		Let $J\subseteq \F_{\mathrm{m}}(\g)$, and let $K$ be a two-sided ideal of
		$(\F_{\mathrm{m}}(\g),\mu_{\mathrm{m}})$. Suppose that
		$\widehat d(J)\subseteq K$ and $\alpha_{\mathrm{m}}(J)\subseteq K$. Then
		$\widehat d(\langle J\rangle)\subseteq K$ and
		$\alpha_{\mathrm{m}}(\langle J\rangle)\subseteq K$.
	\end{lemma}
	
	\begin{proof}
		The ideal $\langle J\rangle$ is spanned by parenthesized products
		containing at least one factor from $J$. We proceed by induction on these
		products. If $m\in J$, then
		$\widehat d(m)\in K$ and $\alpha_{\mathrm{m}}(m)\in K$ by assumption.
		
		Suppose that $m=uv$ and that $u$ contains a factor from $J$. By the
		induction hypothesis, $\widehat d(u),\alpha_{\mathrm{m}}(u)\in K$. Since
		$K$ is a two-sided ideal,
		$\widehat d(m)=\widehat d(u)\alpha_{\mathrm{m}}(v)+
		\alpha_{\mathrm{m}}(u)\widehat d(v)\in K$ and
		$\alpha_{\mathrm{m}}(m)=
		\alpha_{\mathrm{m}}(u)\alpha_{\mathrm{m}}(v)\in K$. The case in which
		$v$ contains a factor from $J$ is analogous.
		
		Thus $\widehat d(m),\alpha_{\mathrm{m}}(m)\in K$ for every
		parenthesized product in this spanning set. By linearity,
		$\widehat d(\langle J\rangle)\subseteq K$ and
		$\alpha_{\mathrm{m}}(\langle J\rangle)\subseteq K$.
	\end{proof}
	
	We now impose the Hom-associativity and bracket relations. Let
	$S_{\mathrm{as}}$ be the set of Hom-associativity elements
	$$
	\mathrm{As}(u,v,w):=\alpha_{\mathrm{m}}(u)(vw)-(uv)\alpha_{\mathrm{m}}(w),
	\qquad u,v,w\in \F_{\mathrm{m}}(\g),
	$$
	and let $S_{\mathrm{br}}$ be the set of bracket-commutator elements
	$$
	R(x,y):=j([x,y])-(j(x)j(y)-j(y)j(x)),
	\qquad x,y\in \g.
	$$
	Let $I:=\langle S_{\mathrm{as}}\cup S_{\mathrm{br}}\rangle$ be the
	two-sided ideal of $\F_{\mathrm{m}}(\g)$ generated by these sets. To
	define the twisting map and derivation on the quotient, we first prove
	that $I$ is stable under $\alpha_{\mathrm{m}}$ and $\widehat d$.
	
	\begin{proposition}\label{lem:defining-ideal-stable}
		The ideal $I$ satisfies $\alpha_{\mathrm{m}}(I)\subseteq I$ and $\widehat d(I)\subseteq I.$
	\end{proposition}
	
	\begin{proof}
		Let $S:=S_{\mathrm{as}}\cup S_{\mathrm{br}}$, so that $I=\langle S\rangle$.
		We first show that
		$$
		\widehat d(S)\subseteq I,
		\qquad
		\alpha_{\mathrm{m}}(S)\subseteq I.
		$$
		For the Hom-associativity generators,
		\eqref{eq:dhat-alpha}, \eqref{eq:dhat-leibniz}, and multiplicativity of
		$\alpha_{\mathrm{m}}$ give
		\begin{align*}
			\widehat d(\mathrm{As}(u,v,w))
			&=\mathrm{As}(\widehat d(u),\alpha_{\mathrm{m}}(v),\alpha_{\mathrm{m}}(w))\\
			&\quad+\mathrm{As}(\alpha_{\mathrm{m}}(u),\widehat d(v),\alpha_{\mathrm{m}}(w))\\
			&\quad+\mathrm{As}(\alpha_{\mathrm{m}}(u),\alpha_{\mathrm{m}}(v),\widehat d(w)).
		\end{align*}
		Thus $\widehat d(S_{\mathrm{as}})\subseteq I$. Also,
		$$
		\alpha_{\mathrm{m}}(\mathrm{As}(u,v,w))=\mathrm{As}(\alpha_{\mathrm{m}}(u),\alpha_{\mathrm{m}}(v),\alpha_{\mathrm{m}}(w)),
		$$
		so $\alpha_{\mathrm{m}}(S_{\mathrm{as}})\subseteq I$.
		
		For the bracket relations, the $\alpha$-derivation identity for
		$d$ gives
		\begin{align*}
			\widehat d(R(x,y))
			&=j(d([x,y]))-\widehat d(j(x)j(y)-j(y)j(x))\\
			&=R(d(x),\alpha(y))+R(\alpha(x),d(y)).
		\end{align*}
		Thus $\widehat d(S_{\mathrm{br}})\subseteq I$. Since the twisting maps on
		$\g$ and $\F_{\mathrm{m}}(\g)$ are multiplicative and $j$ is a
		Hom-module morphism,
		$$
		\alpha_{\mathrm{m}}(R(x,y))=R(\alpha(x),\alpha(y)),
		$$
		so $\alpha_{\mathrm{m}}(S_{\mathrm{br}})\subseteq I$.
		
		Applying Lemma~\ref{lem:ideal-closure} with $J=S$ and $K=I$ gives
		$$
		\widehat d(I)\subseteq I,
		\qquad
		\alpha_{\mathrm{m}}(I)\subseteq I.
		$$
	\end{proof}
	
	Let $\U_{\mathrm{m}}(\g):=\F_{\mathrm{m}}(\g)/I$, and let
	$\pi_I:\F_{\mathrm{m}}(\g)\longrightarrow\U_{\mathrm{m}}(\g)$ be the
	quotient map. We denote by $\mu_U$, $\alpha_U$, and $D_U$ the maps induced by
	$\mu_{\mathrm{m}}$, $\alpha_{\mathrm{m}}$, and $\widehat d$,
	respectively. Thus, for $u,v\in \F_{\mathrm{m}}(\g)$,
	\begin{align*}
	\mu_U(\pi_I(u),\pi_I(v))&:=\pi_I(\mu_{\mathrm{m}}(u,v)),\\
	\alpha_U(\pi_I(u))&:=\pi_I(\alpha_{\mathrm{m}}(u)),\\
	D_U(\pi_I(u))&:=\pi_I(\widehat d(u)).
	\end{align*}
	
	\begin{proposition}\label{prop:Um-HomAsDer}
		The quadruple $(\U_{\mathrm{m}}(\g),\mu_U,\alpha_U,D_U)$ is an object of
		$\mathbf{HomAsDer}_\alpha$.
	\end{proposition}
	
	\begin{proof}
		Since $I$ is a two-sided ideal, $\mu_U$ is well defined.
		Proposition~\ref{lem:defining-ideal-stable} gives
		$\alpha_{\mathrm{m}}(I)\subseteq I$ and $\widehat d(I)\subseteq I$,
		so $\alpha_U$ and $D_U$ are also well defined.
		
		The quotient is multiplicative because $\alpha_{\mathrm{m}}$ is
		multiplicative, and it is Hom-associative because
		$\mathrm{As}(u,v,w)\in I$.
		
		To verify the derivation identities, let $a,b\in\U_{\mathrm{m}}(\g)$.
		Since $\pi_I$ is surjective, choose $u,v\in\F_{\mathrm{m}}(\g)$ such that
		$a=\pi_I(u)$ and $b=\pi_I(v)$. By \eqref{eq:dhat-alpha} and the
		definitions of the induced maps, we have
		\begin{align*}
			D_U(\alpha_U(a))
			&=\pi_I\bigl(\widehat d(\alpha_{\mathrm{m}}(u))\bigr)\\
			&=\pi_I\bigl(\alpha_{\mathrm{m}}(\widehat d(u))\bigr)\\
			&=\alpha_U(D_U(a)).
		\end{align*}
		Similarly, \eqref{eq:dhat-leibniz} gives
		\begin{align*}
			D_U(ab)
			&=\pi_I\bigl(\widehat d(uv)\bigr)\\
			&=\pi_I\bigl(\widehat d(u)\alpha_{\mathrm{m}}(v)
			+\alpha_{\mathrm{m}}(u)\widehat d(v)\bigr)\\
			&=D_U(a)\alpha_U(b)+\alpha_U(a)D_U(b).
		\end{align*}
		These equalities hold for all $a,b\in\U_{\mathrm{m}}(\g)$. Thus
		$D_U$ is an $\alpha_U$-derivation, and
		$(\U_{\mathrm{m}}(\g),\mu_U,\alpha_U,D_U)$ is an object of
		$\mathbf{HomAsDer}_\alpha$.
	\end{proof}
	
	Write $\iota_U:=\pi_I\circ j:\g\to\U_{\mathrm{m}}(\g)$ for the canonical map.
	
	\begin{proposition}\label{prop:iota-der}
		Let $(\g,[\cdot,\cdot],\alpha,d)$ be an object of
		$\mathbf{HomLieDer}_\alpha$. Then the canonical map
		$$
		\iota_U: (\g,[\cdot,\cdot],\alpha,d) \to HLie_{\Der_\alpha}(\U_{\mathrm{m}}(\g),\mu_U,\alpha_U,D_U)
		$$
		is a morphism in $\mathbf{HomLieDer}_\alpha$.
	\end{proposition}
	
	\begin{proof}
		The equality $R(x,y)=j([x,y])-(j(x)j(y)-j(y)j(x))\in I$ shows that $\iota_U$ preserves the
		bracket. The quotient construction
		also gives compatibility with the twisting maps. Finally, for $x\in\g$,
		$$
		D_U(\iota_U(x))=D_U(\pi_I(j(x)))=\pi_I(\widehat d(j(x)))=\pi_I(j(d(x)))=\iota_U(d(x)).
		$$
		Thus $\iota_Ud=D_U\iota_U$, completing the proof.
	\end{proof}
	
	\subsection{Construction of the universal enveloping algebra of Hom-Lie algebras with \texorpdfstring{$\alpha$}{alpha}-derivations}
	
	\begin{definition}
		Let $(\g,[\cdot,\cdot],\alpha,d)$ be an object of
		$\mathbf{HomLieDer}_\alpha$. An object $(U,\mu_U,\alpha_U,D_U)$ of
		$\mathbf{HomAsDer}_\alpha$ is called a \textbf{universal enveloping
		algebra} of $(\g,[\cdot,\cdot],\alpha,d)$ if there exists a morphism
		$$
		\iota:(\g,[\cdot,\cdot],\alpha,d)\to
		HLie_{\Der_\alpha}(U,\mu_U,\alpha_U,D_U)
		$$
		in $\mathbf{HomLieDer}_\alpha$ with the following universal property: for
		every object $(A,\mu_A,\alpha_A,D_A)$ of
		$\mathbf{HomAsDer}_\alpha$ and every morphism
		$$
		f:(\g,[\cdot,\cdot],\alpha,d)\to
		HLie_{\Der_\alpha}(A,\mu_A,\alpha_A,D_A)
		$$
		in $\mathbf{HomLieDer}_\alpha$, there exists a unique morphism
		$h:(U,\mu_U,\alpha_U,D_U)\to(A,\mu_A,\alpha_A,D_A)$ in
		$\mathbf{HomAsDer}_\alpha$ such that $f=h\circ\iota$. Equivalently, the
		following diagram commutes:
		\begin{center}
			\small
			\begin{tikzpicture}[
				>=stealth,
				every node/.style={inner sep=2pt}
				]
				\node (g) at (0,2.4) {$({\g},[\cdot,\cdot],\alpha,d)$};
				\node (U) at (0,0) {$(U,\mu_U,\alpha_U,D_U)$};
				\node (A) at (6.5,1.2) {$(A,\mu_A,\alpha_A,D_A)$};
				
				\draw[->] (g) -- node[left] {$\iota$} (U);
				\draw[->] (g.east) -- node[above] {$f$} (A.north west);
				\draw[->,dashed] (U.east) -- node[midway,below,yshift=-5pt] {$\exists!\,h$} (A.south west);
			\end{tikzpicture}
		\end{center}
	\end{definition}
	
	\begin{theorem}\label{thm:main}
		Let $(\g,[\cdot,\cdot],\alpha,d)$ be an object of
		$\mathbf{HomLieDer}_\alpha$. Then
		$(\U_{\mathrm{m}}(\g),\mu_U,\alpha_U,D_U)$ is a universal enveloping
		algebra of $(\g,[\cdot,\cdot],\alpha,d)$. Moreover, any other universal
		enveloping algebra
		$(\U_{\mathrm{m}}'(\g),\mu'_U,\alpha'_U,D'_U)$ of this object satisfies
		$$
		(\U_{\mathrm{m}}'(\g),\mu'_U,\alpha'_U,D'_U)\cong(\U_{\mathrm{m}}(\g),\mu_U,\alpha_U,D_U)
		$$
		in $\mathbf{HomAsDer}_\alpha$.
	\end{theorem}
	
	\begin{proof}
		Let $(A,\mu_A,\alpha_A,D_A)$ be an object of
		$\mathbf{HomAsDer}_\alpha$, and let
		\[
		f:(\g,[\cdot,\cdot],\alpha,d)\to
		HLie_{\Der_\alpha}(A,\mu_A,\alpha_A,D_A)
		\]
		be a morphism in
		$\mathbf{HomLieDer}_\alpha$.
		We prove that $f$ factors uniquely as indicated in the following diagram.
		\begin{figure}[H]
			\centering
			\small
			\begin{tikzpicture}[>=stealth, line width=0.45pt, every node/.style={inner sep=2pt}]
				\node (g) at (0,3) {$({\g},[\cdot,\cdot],\alpha,d)$};
				\node (F) at (0,1.5) {$(\F_{\mathrm{m}}({\g}),\mu_{\mathrm{m}},\alpha_{\mathrm{m}},\widehat d)$};
				\node (U) at (0,0) {$(\U_{\mathrm{m}}({\g}),\mu_U,\alpha_U,D_U)$};
				\node (A) at (7.4,1.5) {$(A,\mu_A,\alpha_A,D_A)$};
				
				\draw[->] (g) -- node[left] {$j$} (F);
				\draw[->] (F) -- node[left] {$\pi_I$} (U);
				\draw[->] (g.east) -- node[above] {$f$} (A.north west);
				\draw[->] (F.east) -- node[above] {$G$} (A.west);
				\node at (3.7,1.26) {\scriptsize Proposition~\ref{prop:free-mult-HNA}};
				\draw[->,dashed] (U.east) -- node[midway,below,yshift=-5pt] {$\exists!\,h$} (A.south west);
			\end{tikzpicture}
			\caption{The universal property of $(\U_{\mathrm{m}}(\g),\mu_U,\alpha_U,D_U)$.}
			\label{fig:universal-property-uea}
		\end{figure}
		
		The universal property of the free Hom-nonassociative algebra gives
		a unique morphism of Hom-nonassociative algebras
		$G:\F_{\mathrm{m}}(\g)\to A$ such that $G\circ j=f$.
		
		We first show that $G(I)=0$. Since $A$ is Hom-associative,
		$$
		G(\mathrm{As}(u,v,w))=\alpha_A(G(u))(G(v)G(w))-(G(u)G(v))\alpha_A(G(w))=0.
		$$
		Since $f$ is a Hom-Lie morphism into the commutator Hom-Lie algebra of $A$,
		$$
		G(R(x,y))=f([x,y])-(f(x)f(y)-f(y)f(x))=0.
		$$
		Thus $G$ vanishes on the generating set
		$S_{\mathrm{as}}\cup S_{\mathrm{br}}$ and hence on
		$I=\langle S_{\mathrm{as}}\cup S_{\mathrm{br}}\rangle$. Since $G(I)=0$,
		there exists a unique Hom-nonassociative algebra morphism $h:\U_{\mathrm{m}}(\g)\to A$ satisfying
		$G=h\circ \pi_I$.
		Both its source and target are Hom-associative, so $h$ is a morphism of
		Hom-associative algebras. Moreover, $f=h\circ\iota_U$.
		
		To prove that $h$ is compatible with the derivations, we consider the maps
		$\Theta_1,\Theta_2:\F_{\mathrm{m}}(\g)\to A\oplus A$ defined by
		$$
		\Theta_1(u):=(G(u),G(\widehat d(u))),
		\qquad
		\Theta_2(u):=(G(u),D_A(G(u))).
		$$
		Let $\widetilde\alpha_A$ and $\star$ denote the twisting map and product,
		respectively, on $A\oplus A$ from
		Proposition~\ref{prop:semidirect-HNA}. We show that both $\Theta_1$ and
		$\Theta_2$ are morphisms of Hom-nonassociative algebras.
		Since $G$ preserves the product and the twisting map,
		\eqref{eq:dhat-leibniz} gives, for $u,v\in\F_{\mathrm{m}}(\g)$,
		\begin{align*}
		\Theta_1(uv)
		&=(G(uv),G(\widehat d(uv)))\\
		&=(G(u)G(v), G(\widehat d(u))\alpha_A(G(v)) +\alpha_A(G(u))G(\widehat d(v)))\\
		&=\Theta_1(u)\star\Theta_1(v).
		\end{align*}
		Moreover, by \eqref{eq:dhat-alpha},
		\begin{align*}
		\Theta_1(\alpha_{\mathrm{m}}(u))
		&=(G(\alpha_{\mathrm{m}}(u)),G(\widehat d(\alpha_{\mathrm{m}}(u))))\\
		&=(\alpha_A(G(u)),\alpha_A(G(\widehat d(u))))\\
		&=\widetilde\alpha_A(\Theta_1(u)).
		\end{align*}
		Thus $\Theta_1$ is a morphism of Hom-nonassociative algebras.
		Since $D_A$ is an $\alpha_A$-derivation of $A$, we also have
		\begin{align*}
		\Theta_2(uv)
		&=(G(uv),D_A(G(uv)))\\
		&=(G(u)G(v),D_A(G(u))\alpha_A(G(v))+\alpha_A(G(u))D_A(G(v)))\\
		&=\Theta_2(u)\star\Theta_2(v),
		\end{align*}
		and the identity $D_A\alpha_A=\alpha_A D_A$ gives
		\begin{align*}
		\Theta_2(\alpha_{\mathrm{m}}(u))
		&=(G(\alpha_{\mathrm{m}}(u)),D_A(G(\alpha_{\mathrm{m}}(u))))\\
		&=(\alpha_A(G(u)),\alpha_A(D_A(G(u))))\\
		&=\widetilde\alpha_A(\Theta_2(u)).
		\end{align*}
		Hence $\Theta_2$ is a morphism of Hom-nonassociative algebras.
		The two morphisms agree on the generators: for every $x\in \g$,
		$$
		\Theta_1(j(x))=(f(x),f(d(x)))=\Theta_2(j(x)),
		$$
		because $f\circ d=D_A\circ f$. By uniqueness in the universal property of
		$\F_{\mathrm{m}}(\g)$, we obtain $\Theta_1=\Theta_2$. Hence $G\circ\widehat d=D_A\circ G$. For
		$\bar u=\pi_I(u)\in \U_{\mathrm{m}}(\g)$, this gives
		$$
		h(D_U(\bar u))=h(\pi_I(\widehat d(u)))=G(\widehat d(u))=D_A(G(u))=D_A(h(\bar u)).
		$$
		Therefore, $h\circ D_U=D_A\circ h$, so $h$ is a morphism in
		$\mathbf{HomAsDer}_\alpha$.
		
		Finally, suppose that $h'$ is another morphism in
		$\mathbf{HomAsDer}_\alpha$ satisfying $h'\circ \iota_U=f$. Then
		$h'\circ\pi_I:\F_{\mathrm{m}}(\g)\to A$ is a morphism of
		Hom-nonassociative algebras extending $f$. By the uniqueness of $G$,
		$h'\circ\pi_I=G=h\circ\pi_I$. Since $\pi_I$ is surjective, $h'=h$.
		
		To prove uniqueness up to isomorphism, let
		$(\U_{\mathrm{m}}'(\g),\mu'_U,\alpha'_U,D'_U)$ be another
		universal enveloping algebra of
		$(\g,[\cdot,\cdot],\alpha,d)$, with canonical morphism
		$\iota'_U:\g\to\U_{\mathrm{m}}'(\g)$. The two universal properties give
		morphisms $\Phi:\U_{\mathrm{m}}(\g)\to\U_{\mathrm{m}}'(\g)$ and
		$\Psi:\U_{\mathrm{m}}'(\g)\to\U_{\mathrm{m}}(\g)$ in $\mathbf{HomAsDer}_\alpha$ such that
		$\Phi\circ\iota_U=\iota'_U$ and $\Psi\circ\iota'_U=\iota_U$. Hence
		$$
		(\Psi\circ\Phi)\circ\iota_U=\iota_U,
		\qquad
		(\Phi\circ\Psi)\circ\iota'_U=\iota'_U.
		$$
		By uniqueness in the respective universal properties,
		$$
		\Psi\circ\Phi=\id_{\U_{\mathrm{m}}(\g)},
		\qquad
		\Phi\circ\Psi=\id_{\U_{\mathrm{m}}'(\g)}.
		$$
		Therefore, $\Phi$ is an isomorphism in $\mathbf{HomAsDer}_\alpha$.
	\end{proof}

	We now use the universal property to define the enveloping functor and
	establish its adjunction with the commutator functor.
	
	\begin{corollary}\label{cor:functor-adjunction}
		There is a functor
		$$
		\U_{\mathrm{m},\Der_\alpha}:\mathbf{HomLieDer}_\alpha\to \mathbf{HomAsDer}_\alpha,
		\qquad
		(\g,[\cdot,\cdot],\alpha,d)
		\mapsto
		(\U_{\mathrm{m}}(\g),\mu_U,\alpha_U,D_U),
		$$
		which is left adjoint to
		$HLie_{\Der_\alpha}:\mathbf{HomAsDer}_\alpha\to
		\mathbf{HomLieDer}_\alpha$.
	\end{corollary}
	
	\begin{proof}
		Let $f:(\g,[\cdot,\cdot],\alpha,d)\to(\h,[\cdot,\cdot]_{\h},\beta,\delta)$ be a morphism in
		$\mathbf{HomLieDer}_\alpha$. We denote the canonical maps by $\iota_{\g}:\g\to
		\U_{\mathrm{m}}(\g)$ and $\iota_{\h}:\h\to \U_{\mathrm{m}}(\h)$. By Theorem~\ref{thm:main},
		applied to the morphism
		$\iota_{\h}\circ f$, there exists a unique morphism
		$\U_{\mathrm{m},\Der_\alpha}(f):\U_{\mathrm{m}}(\g)\to \U_{\mathrm{m}}(\h)$ in
		$\mathbf{HomAsDer}_\alpha$ satisfying
		$\U_{\mathrm{m},\Der_\alpha}(f)\circ\iota_{\g}=\iota_{\h}\circ f$. The same uniqueness statement
		gives the identity and composition laws:
		in each case, the two morphisms agree after composition with the canonical
		map. Thus the assignment defines a functor.
		
		Theorem~\ref{thm:main} gives a bijection, natural in
		$(\g,[\cdot,\cdot],\alpha,d)$ and $(A,\mu_A,\alpha_A,D_A)$, between
		morphisms
		$$
		\U_{\mathrm{m},\Der_\alpha}(\g,[\cdot,\cdot],\alpha,d)\to (A,\mu_A,\alpha_A,D_A)
		$$
		in $\mathbf{HomAsDer}_\alpha$ and morphisms
		$$
		(\g,[\cdot,\cdot],\alpha,d)\to HLie_{\Der_\alpha}(A,\mu_A,\alpha_A,D_A)
		$$
		in $\mathbf{HomLieDer}_\alpha$. Therefore,
		$\U_{\mathrm{m},\Der_\alpha}$ is left adjoint to $HLie_{\Der_\alpha}$.
	\end{proof}

\section{The PBW theorem for regular Hom-Lie algebras with an \texorpdfstring{$\alpha$}{alpha}-derivation}
	
	Throughout this section, the twisting map is assumed to be bijective.
	We first untwist the bracket or product and the $\alpha$-derivation to
	obtain Lie and associative algebras with derivations. These constructions,
	which form the basis of the comparison and PBW theorems, are given in
	Propositions~\ref{prop:reg-lie} and~\ref{prop:reg-ass}.
	
	\subsection{Lie and associative algebras with derivations}

	Recall that a derivation of a Lie algebra $(\g,[\cdot,\cdot]_0)$ is a
	linear map $\delta:\g\to\g$ satisfying
	$$
	\delta([x,y]_0)
	=
	[\delta(x),y]_0+[x,\delta(y)]_0
	\qquad \forall x,y\in\g.
	$$
	
	\begin{proposition}\label{prop:reg-lie}
		Let $(\g,[\cdot,\cdot],\alpha,d)$ be a regular Hom-Lie algebra with an
		$\alpha$-derivation $d$. Then $(\g,\alpha^{-1}\circ[\cdot,\cdot])$
		is a Lie algebra, and $\alpha^{-1}\circ d$ is a derivation of this algebra.
	\end{proposition}
	
	\begin{proof}
		Set $[x,y]_0:=\alpha^{-1}([x,y])$ and $\delta:=\alpha^{-1}\circ d$.
		The bracket $[\cdot,\cdot]_0$ is skew-symmetric because
		$[\cdot,\cdot]$ is skew-symmetric. Since $\alpha$ is multiplicative and
		bijective, we have
		$$
		[x,\alpha^{-1}(u)]=\alpha^{-1}([\alpha(x),u])
		\qquad \forall x,u\in\g.
		$$
		Using this identity and the Hom-Jacobi identity, we obtain, for all $x,y,z\in\g$,
		\begin{align*}
		&[x,[y,z]_0]_0+[y,[z,x]_0]_0+[z,[x,y]_0]_0\\
		&\quad=\alpha^{-2}\bigl([\alpha(x),[y,z]]+[\alpha(y),[z,x]]+[\alpha(z),[x,y]]\bigr)=0.
		\end{align*}
		Thus $(\g,[\cdot,\cdot]_0)$ is a Lie algebra.
		
		To show that $\delta$ is a derivation, let $x,y\in\g$. Then
		\begin{align*}
			\delta([x,y]_0)
			&=
			\alpha^{-1}d(\alpha^{-1}([x,y]))\\
			&=
			\alpha^{-2}d([x,y])
			\qquad (\text{since } d\alpha=\alpha d)\\
			&=
			\alpha^{-2}([d(x),\alpha(y)]+[\alpha(x),d(y)])\\
			&=
			\alpha^{-1}([\alpha^{-1}d(x),y])+\alpha^{-1}([x,\alpha^{-1}d(y)])\\
			&=
			[\delta(x),y]_0+[x,\delta(y)]_0.
		\end{align*}
		Thus $\delta$ is a derivation, as required.
	\end{proof}
	
	\begin{remark}\label{rem:regular-lie-untwisted}
		With the notation used in the proof of Proposition~\ref{prop:reg-lie}, $\alpha$ is an
		automorphism of the Lie algebra $\g_0=(\g,[\cdot,\cdot]_0)$, since
		$[\alpha(x),\alpha(y)]_0
		=\alpha^{-1}([\alpha(x),\alpha(y)])
		=\alpha([x,y]_0).$
		Moreover, $\delta=\alpha^{-1}\circ d$ and
		$d\circ\alpha=\alpha\circ d$ imply
		$\alpha\circ\delta=\delta\circ\alpha=d$. Consequently,
		$[x,y]=\alpha([x,y]_0)$ and $d=\alpha\circ\delta$.
	\end{remark}
	
	The same untwisting procedure applies to Hom-associative algebras.
	
	\begin{proposition}\label{prop:reg-ass}
		Let $(A,\mu,\alpha,D)$ be a regular Hom-associative algebra with an
		$\alpha$-derivation $D$. Then $(A,\alpha^{-1}\circ\mu)$ is an
		associative algebra, and $\alpha^{-1}\circ D$ is a derivation of this algebra.
	\end{proposition}

	\begin{proof}
		Set $a\ast b:=\alpha^{-1}(ab)$ and $\Delta:=\alpha^{-1}\circ D$.
		For $a,b,c\in A$, multiplicativity of $\alpha$ gives
		\begin{align*}
		\alpha^2((a\ast b)\ast c)&=(ab)\alpha(c),\\
		\alpha^2(a\ast(b\ast c))&=\alpha(a)(bc).
		\end{align*}
		Hom-associativity makes the right-hand sides equal. Since $\alpha$ is
		bijective, the product $\ast$ is associative.
		
		For the derivation property, let $a,b\in A$. We have
		\begin{align*}
			\Delta(a\ast b)
			&=
			\alpha^{-1}D(\alpha^{-1}(ab))
			=
			\alpha^{-2}D(ab)\\
			&=
			\alpha^{-2}(D(a)\alpha(b)+\alpha(a)D(b))\\
			&=
			\alpha^{-1}(D(a))\ast b+a\ast \alpha^{-1}(D(b))\\
			&=
			\Delta(a)\ast b+a\ast \Delta(b).
		\end{align*}
		Thus $\Delta$ is a derivation.
	\end{proof}
	
	\begin{remark}\label{rem:regular-associative-untwisted}
		With the notation used in the proof of Proposition~\ref{prop:reg-ass}, $\alpha$ is an
		automorphism of the associative algebra $(A,\ast)$, since
		$\alpha(a\ast b)
		=ab
		=\alpha(a)\ast\alpha(b).$
		Moreover, $\Delta=\alpha^{-1}\circ D$ and
		$D\circ\alpha=\alpha\circ D$ imply
		$\alpha\circ\Delta=\Delta\circ\alpha=D$. Consequently,
		$ab=\alpha(a\ast b)$ and $D=\alpha\circ\Delta$.
	\end{remark}

	\begin{remark}
		By Propositions~\ref{prop:reg-lie} and~\ref{prop:reg-ass}, $\delta$ is a
		derivation of the Lie algebra $\g_0$, and $\Delta$ is a derivation of the
		associative algebra $(A,\ast)$. In particular, $(\g_0,\delta)$ is a
		$\lambda$-differential Lie algebra in the sense of
		\cite[Definition~2.5]{PengZhangGaoLuo} for $\lambda=0$ and the zero
		derivation on the ground field, since
		$\delta([x,y]_0)
		=[\delta(x),y]_0+[x,\delta(y)]_0.$
		Similarly, $(A,\ast,\Delta)$ satisfies
		$
		\Delta(a\ast b)
		=\Delta(a)\ast b+a\ast\Delta(b).
		$
	\end{remark}

	\subsection{Comparison with the twisted nonunital enveloping algebra}

	We first show that the enveloping construction preserves regularity.
	
	\begin{lemma}\label{lem:regularity-envelope}
		Let $(\g,[\cdot,\cdot],\alpha,d)$ be a regular Hom-Lie algebra with an
		$\alpha$-derivation, and let
		$(\U_{\mathrm{m}}(\g),\mu_U,\alpha_U,D_U)$ be its universal enveloping
		Hom-associative algebra. Then $\alpha_U$ is bijective.
	\end{lemma}

	\begin{proof}
		We first verify that $\alpha^{-1}:\g\to\g$ is a morphism in
		$\mathbf{HomLieDer}_\alpha$. By multiplicativity and the bijectivity
		of $\alpha$, we have
		$$
		\alpha([\alpha^{-1}(x),\alpha^{-1}(y)])=[x,y],
		$$
		and hence
		$$
		[\alpha^{-1}(x),\alpha^{-1}(y)]=\alpha^{-1}([x,y]).
		$$
		Moreover, $\alpha^{-1}$ commutes with both $\alpha$ and $d$, since
		$d\alpha=\alpha d$.
		
		Consider the morphism
		$$
		\iota_U\circ\alpha^{-1}:\g\to
		HLie_{\Der_\alpha}(\U_{\mathrm{m}}(\g),\mu_U,\alpha_U,D_U).
		$$
		By the universal property of $\U_{\mathrm{m}}(\g)$, there exists a unique morphism
		$$
		T:(\U_{\mathrm{m}}(\g),\mu_U,\alpha_U,D_U)\to
		(\U_{\mathrm{m}}(\g),\mu_U,\alpha_U,D_U)
		$$
		in $\mathbf{HomAsDer}_\alpha$ such that $T\circ\iota_U=\iota_U\circ\alpha^{-1}$. On the other
		hand, $\alpha_U$ is an endomorphism of
		$(\U_{\mathrm{m}}(\g),\mu_U,\alpha_U,D_U)$ in
		$\mathbf{HomAsDer}_\alpha$, because it preserves the product and
		$D_U\alpha_U=\alpha_U D_U$. Moreover, $\alpha_U\circ\iota_U=\iota_U\circ\alpha$. Therefore
		$$
		(\alpha_U\circ T)\circ\iota_U=\iota_U,
		\qquad
		(T\circ\alpha_U)\circ\iota_U=\iota_U.
		$$
		The uniqueness in the universal property applied to $\iota_U$ implies that both
		$\alpha_U\circ T$ and $T\circ\alpha_U$ are the identity on
		$\U_{\mathrm{m}}(\g)$. Hence $T=\alpha_U^{-1}$, and $\alpha_U$ is
		bijective.
	\end{proof}

	We now construct the classical enveloping algebra to be used in the comparison.
	Let $(\g,[\cdot,\cdot],\alpha,d)$ be a regular Hom-Lie algebra with an
	$\alpha$-derivation. For $x,y\in\g$, set
	$[x,y]_0:=\alpha^{-1}([x,y])$ and
	$\delta(x):=\alpha^{-1}(d(x))$. We denote the resulting Lie algebra
	$(\g,[\cdot,\cdot]_0)$ by $\g_0$. By Proposition~\ref{prop:reg-lie},
	$\delta$ is a derivation of $\g_0$.
	
	Let $T^+(\g):=\bigoplus_{n\geq 1}\g^{\otimes n}$,
	and let $J_0$ be the two-sided ideal generated by
	$x\otimes y-y\otimes x-[x,y]_0$, $x,y\in\g$.
	Define the nonunital universal enveloping algebra by
	$
	\U^+(\g_0):=T^+(\g)/J_0,
	$
	and let $\iota_0:\g_0\to\U^+(\g_0)$ be the canonical map. For every
	associative algebra $B$ and every Lie algebra morphism
	$f:\g_0\to B^{-}$, where $B^{-}$ denotes the commutator Lie algebra of $B$,
	there exists a unique algebra morphism
	$\bar f:\U^+(\g_0)\to B$ such that
	$\bar f\circ\iota_0=f$. Here the associative algebras are not required to
	have units, and the morphisms are not required to preserve units.
	
	The maps $\alpha$ and $\delta$ extend uniquely to an algebra automorphism
	$\widetilde{\alpha}$ and a derivation $\widetilde{\delta}$ of
	$\U^+(\g_0)$ such that
	$
	\widetilde{\alpha}\circ\iota_0=\iota_0\circ\alpha
	$
	and
	$
	\widetilde{\delta}\circ\iota_0=\iota_0\circ\delta.
	$
	They satisfy
	$
	\widetilde{\alpha}\widetilde{\delta}
	=\widetilde{\delta}\widetilde{\alpha}.
	$
	Define
	$
	a\cdot_\alpha b:=\widetilde{\alpha}(ab)
	$
	and
	$
	\widetilde d:=\widetilde{\alpha}\circ\widetilde{\delta},
	$
	and write
	$
	\U^+(\g_0)^\alpha
	:=(\U^+(\g_0),\cdot_\alpha,\widetilde{\alpha},\widetilde d).
	$
	
	\begin{proposition}\label{prop:classical-yau-twist}
		With the notation above, $\U^+(\g_0)^\alpha$ is an object of
		$\mathbf{HomAsDer}_\alpha$, and $\iota_0$ is a morphism
		$$
		(\g,[\cdot,\cdot],\alpha,d)
		\to HLie_{\Der_\alpha}(\U^+(\g_0)^\alpha).
		$$
	\end{proposition}
	
	\begin{proof}
		The universal property of $\U^+(\g_0)$ follows from the freeness of
		$T^+(\g)$ as a nonunital associative algebra and the defining
		bracket--commutator relations in $J_0$.
		
		By Proposition~\ref{prop:reg-lie}, $(\g,[\cdot,\cdot]_0)$ is a Lie
		algebra and $\delta$ is a derivation. Moreover, $\alpha$ is an automorphism
		of this Lie algebra, since $[\alpha(x),\alpha(y)]_0=\alpha([x,y]_0)$ for all $x,y\in\g$. The
		universal property of $\U^+(\g_0)$ therefore yields an algebra
		endomorphism $\widetilde{\alpha}$ extending $\alpha$. The extension of
		$\alpha^{-1}$ is its inverse, so $\widetilde{\alpha}$ is an
		algebra automorphism.
		
		The Leibniz rule extends $\delta$ to a derivation of $T^+(\g)$. For a
		generator $r(x,y):=x\otimes y-y\otimes x-[x,y]_0$ of $J_0$, we have
		\begin{align*}
			\delta(r(x,y))
			&=
			\delta(x)\otimes y+x\otimes\delta(y)-\delta(y)\otimes x-y\otimes\delta(x)-\delta([x,y]_0)\\
			&=
			r(\delta(x),y)+r(x,\delta(y)),
		\end{align*}
		since $\delta$ is a derivation of $\g_0$. Hence $J_0$ is stable under the
		extended derivation, which descends to a derivation
		$\widetilde{\delta}$ of $\U^+(\g_0)$ satisfying
		$\widetilde{\delta}\iota_0=\iota_0\delta$.
		
		Since $\alpha\delta=\delta\alpha$ on $\g$, the derivation
		$\widetilde{\alpha}^{-1}\widetilde{\delta}\widetilde{\alpha}$ also extends
		$\delta$. By uniqueness, it equals $\widetilde{\delta}$. Hence
		$\widetilde{\alpha}\widetilde{\delta}=
		\widetilde{\delta}\widetilde{\alpha}$.
		
		The product $a\cdot_\alpha b=\widetilde{\alpha}(ab)$ is the twist of the
		associative algebra $\U^+(\g_0)$ by the automorphism
		$\widetilde{\alpha}$ \cite{HomAlgebraTwists}. Therefore,
		$(\U^+(\g_0),\cdot_\alpha,\widetilde{\alpha})$ is a Hom-associative
		algebra. Moreover, for $a,b\in\U^+(\g_0)$, the commutation relation
		$\widetilde{\alpha}\widetilde{\delta}=
		\widetilde{\delta}\widetilde{\alpha}$ gives
		\begin{align*}
			\widetilde d(a\cdot_\alpha b)
			&=
			\widetilde{\alpha}\widetilde{\delta}(\widetilde{\alpha}(ab))\\
			&=
			\widetilde{\alpha}^2(\widetilde{\delta}(a)b+a\widetilde{\delta}(b))\\
			&=
			\widetilde d(a)\cdot_\alpha\widetilde{\alpha}(b)+\widetilde{\alpha}(a)\cdot_\alpha\widetilde d(b).
		\end{align*}
		Thus $\widetilde d$ is a $\widetilde{\alpha}$-derivation.
		
		It remains to show that $\iota_0$ preserves the bracket, twisting map,
		and derivation. For
		$x,y\in\g$,
		\begin{align*}
			[\iota_0(x),\iota_0(y)]_{\U^+(\g_0)^\alpha}
			&=
			\iota_0(x)\cdot_\alpha\iota_0(y)-\iota_0(y)\cdot_\alpha\iota_0(x)\\
			&=
			\widetilde{\alpha}(\iota_0(x)\iota_0(y)-\iota_0(y)\iota_0(x))\\
			&=
			\widetilde{\alpha}(\iota_0([x,y]_0))\\
			&=
			\iota_0(\alpha([x,y]_0))
			=
			\iota_0([x,y]).
		\end{align*}
		The identities $\widetilde{\alpha}\iota_0=\iota_0\alpha$ and
		$\widetilde d\iota_0=\iota_0d$ follow from the definitions and
		$d=\alpha\delta$. Hence $\iota_0$ is a morphism in
		$\mathbf{HomLieDer}_\alpha$.
	\end{proof}
	
	For a regular Hom-Lie algebra $(\g,[\cdot,\cdot],\alpha,d)$ with an
	$\alpha$-derivation, Lemma~\ref{lem:regularity-envelope} ensures that
	$\alpha_U$ is bijective. We use the notation $u\ast v:=\alpha_U^{-1}(uv)$
	and $\Delta_U:=\alpha_U^{-1}\circ D_U$ for the untwisted product and map on
	$\U_{\mathrm{m}}(\g)$.
	
	\begin{theorem}\label{thm:regular-comparison}
		Let $(\g,[\cdot,\cdot],\alpha,d)$ be an object of
		$\mathbf{HomLieDer}_\alpha$ with $\alpha$ bijective. With the notation of
		Proposition~\ref{prop:classical-yau-twist}, there exists a unique
		isomorphism
		$$
		\Phi:(\U_{\mathrm{m}}(\g),\mu_U,\alpha_U,D_U)\to
		\U^+(\g_0)^\alpha
		$$
		in $\mathbf{HomAsDer}_\alpha$ such that $\Phi\circ\iota_U=\iota_0$. Consequently, untwisting the
		product and derivation on
		$\U_{\mathrm{m}}(\g)$ by $\alpha_U^{-1}$ gives an isomorphism of
		associative algebras with derivations:
		$$
		(\U_{\mathrm{m}}(\g),\ast,\Delta_U)\cong(\U^+(\g_0),\widetilde{\delta}).
		$$
	\end{theorem}
	
	\begin{proof}
		By Proposition~\ref{prop:classical-yau-twist}, $\iota_0$ is a morphism
		from $(\g,[\cdot,\cdot],\alpha,d)$ to the commutator Hom-Lie algebra of
		$\U^+(\g_0)^\alpha$. Hence Theorem~\ref{thm:main} gives a unique morphism
		$\Phi:(\U_{\mathrm{m}}(\g),\mu_U,\alpha_U,D_U)\to \U^+(\g_0)^\alpha$ in $\mathbf{HomAsDer}_\alpha$
		satisfying $\Phi\iota_U=\iota_0$.
		
		To construct the inverse, note that $\alpha_U$ is bijective by
		Lemma~\ref{lem:regularity-envelope}. Thus
		Proposition~\ref{prop:reg-ass} applied to
		$(\U_{\mathrm{m}}(\g),\mu_U,\alpha_U,D_U)$ gives the associative algebra
		with derivation $\U_{\mathrm{m}}(\g)^\sharp:=(\U_{\mathrm{m}}(\g),\ast,\Delta_U)$, where $u\ast
		v=\alpha_U^{-1}(uv)$ and $\Delta_U=\alpha_U^{-1}D_U$.
		
		The canonical map $\iota_U:\g\to\U_{\mathrm{m}}(\g)^\sharp$ is a
		morphism from the Lie algebra with derivation $(\g_0,\delta)$ to the
		commutator Lie algebra of $\U_{\mathrm{m}}(\g)^\sharp$. Indeed,
		\begin{align*}
			\iota_U([x,y]_0)
			&=
			\alpha_U^{-1}(\iota_U([x,y]))\\
			&=
			\alpha_U^{-1}(\iota_U(x)\iota_U(y)-\iota_U(y)\iota_U(x))\\
			&=
			\iota_U(x)\ast\iota_U(y)-\iota_U(y)\ast\iota_U(x),
		\end{align*}
		and
		$$
		\Delta_U\iota_U(x)=\alpha_U^{-1}D_U\iota_U(x)=\alpha_U^{-1}\iota_U(d(x))=\iota_U(\delta(x)).
		$$
		By the universal property of $\U^+(\g_0)$, there exists a unique algebra morphism
		$\Psi_0:\U^+(\g_0)\to \U_{\mathrm{m}}(\g)^\sharp$ such that $\Psi_0\iota_0=\iota_U$.
		We next check that this morphism preserves the derivations. Both
		$\Psi_0\widetilde{\delta}$ and $\Delta_U\Psi_0$ are derivations from
		$\U^+(\g_0)$ to the $\U^+(\g_0)$-bimodule
		$\U_{\mathrm m}(\g)^\sharp$ induced by $\Psi_0$, and they
		have the same values on the generators:
		$$
		\Psi_0\widetilde{\delta}\iota_0(x)=\Psi_0\iota_0(\delta(x))=\iota_U(\delta(x))=\Delta_U\iota_U(x)=\Delta_U\Psi_0\iota_0(x).
		$$
		Since $\U^+(\g_0)$ is generated as an associative algebra by
		$\iota_0(\g)$, it follows that
		$\Psi_0\widetilde{\delta}=\Delta_U\Psi_0$.
		
		We now check compatibility with the twisting maps. Since
		$\U^+(\g_0)$ is generated as an associative algebra by $\iota_0(\g)$, the
		equality $\Psi_0\widetilde{\alpha}=\alpha_U\Psi_0$ follows by checking it on $\iota_0(\g)$:
		$$
		\Psi_0\widetilde{\alpha}\iota_0(x)=\Psi_0\iota_0(\alpha(x))=\iota_U(\alpha(x))=\alpha_U\iota_U(x)=\alpha_U\Psi_0\iota_0(x).
		$$
		Therefore, for $a,b\in\U^+(\g_0)$,
		\begin{align*}
			\Psi_0(a\cdot_\alpha b)
			&=
			\Psi_0(\widetilde{\alpha}(ab))\\
			&=
			\alpha_U(\Psi_0(a)\ast\Psi_0(b))\\
			&=
			\Psi_0(a)\Psi_0(b),
		\end{align*}
		where the last product is the Hom-associative product in $\U_{\mathrm{m}}(\g)$. Moreover,
		$$
		\Psi_0\widetilde d=\Psi_0\widetilde{\alpha}\widetilde{\delta}=\alpha_U\Psi_0\widetilde{\delta}=\alpha_U\Delta_U\Psi_0=D_U\Psi_0.
		$$
		Thus $\Psi_0$ defines a morphism $\Psi:\U^+(\g_0)^\alpha\to
		(\U_{\mathrm{m}}(\g),\mu_U,\alpha_U,D_U)$ in $\mathbf{HomAsDer}_\alpha$.
		
		Finally, $\Psi\Phi$ is an endomorphism of $\U_{\mathrm{m}}(\g)$ in
		$\mathbf{HomAsDer}_\alpha$ and $(\Psi\Phi)\iota_U=\iota_U$. The uniqueness in
		Theorem~\ref{thm:main} gives $\Psi\Phi=\id$.
		For the other composition, untwisting turns $\Phi$ into an
		algebra morphism $\Phi^\sharp:\U_{\mathrm{m}}(\g)^\sharp\to(\U^+(\g_0),\widetilde{\delta})$
		preserving the derivations and satisfying
		$\Phi^\sharp\circ\iota_U=\iota_0$. Hence $\Phi^\sharp\Psi_0$ is an
		endomorphism of $(\U^+(\g_0),\widetilde{\delta})$ extending $\iota_0$.
		By the universal property of $\U^+(\g_0)$,
		$\Phi^\sharp\Psi_0=\id$, and hence $\Phi\Psi=\id$. Therefore, $\Phi$ is
		an isomorphism. Its untwisted form gives the stated isomorphism of
		associative algebras with derivations.
	\end{proof}

	\subsection{The PBW theorem for regular Hom-Lie algebras with an \texorpdfstring{$\alpha$}{alpha}-derivation}

	Theorem~\ref{thm:regular-comparison} reduces the PBW theorem for the
	untwisted product to the classical PBW theorem, which we recall below.

	Let $L$ be a Lie algebra over $\kk$, and let $\U(L)$ be its unital universal
	enveloping algebra. Define the standard filtration by
	$$
	F_p\U(L):=\operatorname{Im}\left(\bigoplus_{0\le r\le p}L^{\otimes r}\to\U(L)\right)
	\qquad (p\ge 0).
	$$
	We write $S(L)$ for the symmetric algebra of $L$ and
	$S^+(L):=\bigoplus_{p\ge 1}S^p(L)$. The nonunital universal enveloping
	algebra $\U^+(L)$ carries the induced filtration, with $F_0\U^+(L)=0$.
	
	\begin{lemma}[{\cite[Sections~17.3-17.4]{Humphreys}}]\label{thm:pbw-lie}
		The canonical map $L\to\U(L)$ is injective, and there is an
		isomorphism of graded associative algebras
		$$
		S(L)\cong\operatorname{gr}\U(L).
		$$
		If $L$ is finite-dimensional and $x_1<\cdots<x_n$ is an ordered basis, then
		$$
		1,
		\qquad
		x_{i_1}\cdots x_{i_r}
		\quad
		(r\ge 1,\ 1\le i_1\le\cdots\le i_r\le n)
		$$
		form a basis of $\U(L)$. For the nonunital universal enveloping algebra
		$\U^+(L)$, the corresponding graded algebra is given by
		$$
		\operatorname{gr}\U^+(L)\cong S^+(L).
		$$
	\end{lemma}

	Let $(\g,[\cdot,\cdot],\alpha,d)$ be an object of
	$\mathbf{HomLieDer}_\alpha$ with $\alpha$ bijective, and set
	$
	[x,y]_0:=\alpha^{-1}([x,y])
	$
	and
	$
	\delta:=\alpha^{-1}\circ d.
	$
	Then $\g_0=(\g,[\cdot,\cdot]_0)$ is a Lie algebra with derivation $\delta$.
	Define a filtration on $\U^+(\g_0)$ by
	$$
	F_p\U^+(\g_0):=\operatorname{Im}\left(\bigoplus_{1\le r\le p}\g^{\otimes r}\to \U^+(\g_0)\right)
	\qquad (p\ge 1),
	$$
	and set $F_0\U^+(\g_0)=0$. We transport this filtration along the isomorphism
	$\Phi$ of Theorem~\ref{thm:regular-comparison} by defining
	$$
	F_p\U_{\mathrm{m}}(\g):=\Phi^{-1}\left(F_p\U^+(\g_0)\right)
	\qquad (p\ge 0).
	$$

	\begin{theorem}\label{cor:regular-pbw}
		Let $(\g,[\cdot,\cdot],\alpha,d)$ be an object of
		$\mathbf{HomLieDer}_\alpha$ with $\alpha$ bijective. Then the filtration
		$\{F_p\U_{\mathrm{m}}(\g)\}_{p\ge 0}$ is preserved by $\alpha_U$,
		$\Delta_U$, and $D_U$. With respect to the untwisted product $\ast$,
		there is an isomorphism of graded associative algebras
		$$
		\operatorname{gr}\big(\U_{\mathrm{m}}(\g),\ast\big)
		\cong
		S^+(\g_0).
		$$
		If $\g$ is finite-dimensional and $x_1<\cdots<x_n$ is an ordered basis
		of its underlying vector space, then the monomials
		$$
		\iota_U(x_{i_1})\ast\cdots\ast\iota_U(x_{i_r})
		\quad
		(r\ge 1,\ 1\le i_1\le\cdots\le i_r\le n)
		$$
		form a basis of $\U_{\mathrm{m}}(\g)$.
	\end{theorem}

	\begin{proof}
		By Theorem~\ref{thm:regular-comparison}, $\Phi$ induces an isomorphism of
		associative algebras with derivations
		$$
		\Phi^\sharp:
		(\U_{\mathrm{m}}(\g),\ast,\Delta_U)
		\longrightarrow
		(\U^+(\g_0),\widetilde{\delta}),
		$$
		with $\Phi^\sharp\circ\iota_U=\iota_0$. By
		Lemma~\ref{thm:pbw-lie},
		$$
		\operatorname{gr}\U^+(\g_0)\cong S^+(\g_0).
		$$
		Transporting the filtration along $\Phi^\sharp$ therefore gives
		$$
		\operatorname{gr}\big(\U_{\mathrm{m}}(\g),\ast\big)
		\cong S^+(\g_0).
		$$
		If $\g$ is finite-dimensional, $(\Phi^\sharp)^{-1}$ sends the ordered PBW
		monomials in $\U^+(\g_0)$ to the stated monomials formed with $\ast$. Hence the
		latter form a basis of $\U_{\mathrm{m}}(\g)$.

		It remains to prove that the three maps preserve the filtration.
		The map $\widetilde{\alpha}$ sends a product of $r$ generators to a product
		of $r$ generators and therefore preserves the filtration of
		$\U^+(\g_0)$. The Leibniz rule shows that $\widetilde{\delta}$ preserves
		the same filtration. Under $\Phi^\sharp$, these maps correspond to
		$\alpha_U$ and $\Delta_U$, respectively. Hence $\alpha_U$ and $\Delta_U$
		preserve the filtration of $\U_{\mathrm{m}}(\g)$. Since $D_U=\alpha_U\circ\Delta_U$, $D_U$
		preserves the filtration as well.
	\end{proof}

	\begin{remark}\label{rem:pbw-comparison}
		Every involutive Hom-Lie algebra is regular because
		$\alpha^2=\id$ implies that $\alpha$ is bijective. Hence
		Theorem~\ref{cor:regular-pbw} applies to the involutive Hom-Lie
		algebras considered in \cite{GuoZhangZheng} when equipped with
		$\alpha$-derivations. The bijectivity assumption is required for the
		untwisting used in the theorem.
	\end{remark}

	\section{Conclusion and further problems}

	We have constructed universal enveloping Hom-associative algebras for
	multiplicative Hom-Lie algebras with $\alpha$-derivations. The derivation
	extends to the free Hom-nonassociative algebra and descends through the
	Hom-associativity and bracket--commutator relations. The resulting
	universal property gives an enveloping functor left adjoint to the
	commutator functor.

	In the regular case, the enveloping algebra is canonically isomorphic to
	the twist of the nonunital universal enveloping algebra of the associated
	Lie algebra. Untwisting this isomorphism gives the PBW description of the
	associated graded algebra and, for a finite-dimensional Hom-Lie algebra,
	an ordered monomial basis. The twisting map and both induced derivations
	preserve the PBW filtration.

	Further questions concern the representation theory and cohomology of
	Hom-Lie algebras with compatible derivations, and PBW theorems under weaker
	conditions on the twisting map. One may also ask whether the enveloping
	construction extends to Hom-Poisson algebras or graded Hom-Lie algebras
	with compatible derivations.

\end{document}